\documentclass[a4paper,11pt,oneside,reqno]{amsart}
\usepackage{amsmath,amssymb,amsthm,mathtools}
\usepackage[pagewise]{lineno}
\usepackage{graphicx}
\usepackage{color}
\mathtoolsset{showonlyrefs}
\usepackage{bm}

\usepackage{subfigure}
\usepackage{verbatim}
\usepackage{wrapfig}
\usepackage{ascmac}
\usepackage{cases}
\usepackage{amsthm}
\usepackage{autobreak}

\usepackage{comment}
\usepackage {cancel}
\usepackage[normalem]{ulem}

\usepackage{microtype}

  \makeatletter
    
    \@addtoreset{equation}{section}
  \makeatother
\newtheorem{theorem}{Theorem}[section]

\newtheorem{lemma}[theorem]{Lemma}

\newcounter{c}
\newcounter{d}
\newcounter{b}

\newcommand{\Na}{\mathbb{N}} 
\newcommand{\R}{\mathbb{R}} 

\newcommand{\ind}[1]{\mathbf{1}_{#1}}

\begin{document}
\raggedbottom
\title[Persistence of the reaction term]{Persistence of the reaction term in an unbalanced fast reaction limit}
\author{Yuki Tsukamoto}
\date{}

\address{Tokyo University of Science,
	162-8601, Tokyo, Japan}
\email{tsukamoto-yuki@rs.tus.ac.jp} 

\begin{abstract}
We study a fast reaction limit for a two-component reaction--diffusion
system in which one component diffuses and the other does not. In the
regime considered here, the nondiffusing component converges to zero as
the reaction rate tends to infinity, while the reaction term has a
nonzero limit. As a result, the limiting equation for the diffusing
component retains a time-integral term. Our results cover both homogeneous 
Dirichlet and Neumann boundary conditions. We prove convergence of the 
diffusing component and determine the limiting behavior of the reaction 
term and the rescaled nondiffusing component. The proof combines monotonicity with
respect to the reaction rate with barrier estimates adapted to the
time-integral term. Under an additional separation condition on the
initial data, we also prove uniqueness of the limiting initial-boundary
value problem.
\end{abstract}
\maketitle

\section{Introduction}\label{sec:introduction}

Fast reaction limits arise in reaction--diffusion systems in which the
reaction takes place on a much shorter time scale than diffusion. A typical
two-component system has the form
\begin{equation}\label{eq:general-fast-reaction-intro}
\begin{cases}
  \partial_t u_k
  =d_1\Delta u_k-k f(u_k,v_k),
  \\[1mm]
  \partial_t v_k
  =d_2\Delta v_k-k g(u_k,v_k),
\end{cases}
\end{equation}
where $d_1,d_2\ge0$ are diffusion coefficients, $f$ and $g$ describe the
reaction terms, and $k>0$ is the reaction-rate parameter. 
The fast reaction limit concerns the behavior of the solutions as
$k\to\infty$. A characteristic feature of many such problems is the
appearance of spatial segregation between the reacting components.
The resulting interface may move or remain stationary, and in some cases
it may disappear instantaneously.

In this paper, we consider the case in which the first component diffuses
and the second one does not. One example motivating such a structure is
the model of the penetration of radiolabeled antibodies into tumorous tissue
studied by Hilhorst, van der Hout and Peletier
\cite{HilhorstEtAl1996FastReactionLimit}, in which mobile antibodies bind
to immobile antigens in the tissue. The absence of diffusion in the second
component is particularly important here, since its equation can be solved
explicitly at each spatial point. As will be seen below, this pointwise 
representation leads to a time-integral term in the fast reaction limit.

In a basic class of irreversible fast reaction systems, the fast reaction
terms in the two equations are proportional.
After normalization, a representative balanced system
has the form
\[
  \partial_tu_k=d_u\Delta u_k-kF(u_k,v_k),
  \qquad
  \partial_tv_k=d_v\Delta v_k-kF(u_k,v_k),
\]
where $d_u>0$ and $d_v\ge0$. Subtracting the two equations cancels the
reaction terms and gives an equation for the difference of the two
components. Hilhorst, van der Hout and Peletier studied the fast reaction
limit for a system with one diffusive and one nondiffusive component
\cite{HilhorstEtAl1996FastReactionLimit}, and subsequently treated general
monotone reaction terms
\cite{HilhorstVanDerHoutPeletier1997GeneralMonotoneFastReaction}.
Eymard, Hilhorst, van der Hout and Peletier gave a reaction-diffusion
approximation of a one-phase Stefan problem
\cite{EymardEtAl2001ReactionDiffusionOnePhaseStefan}. On unbounded domains,
Crooks and Hilhorst considered systems with one or two diffusing components
and characterized their self-similar fast reaction limits
\cite{CrooksHilhorst2016SelfsimilarFastreactionLimits}.

Related singular limits arise in strongly competitive systems. Dancer,
Hilhorst, Mimura and Peletier proved convergence of a competition-diffusion
system to a Stefan problem with zero latent heat
\cite{DancerEtAl1999SpatialSegregationLimit}. Crooks, Dancer, Hilhorst,
Mimura and Ninomiya studied spatial segregation under inhomogeneous
Dirichlet boundary conditions, with a free boundary problem arising in the
limit
\cite{CrooksEtAl2004SpatialSegregationLimit}.
Conti, Terracini and Verzini obtained asymptotic estimates and regularity
properties for the limiting configurations of strongly competitive systems
\cite{ContiEtAl2005AsymptoticEstimatesSpatial}.
Hilhorst, Iida, Mimura and Ninomiya proved $L^p$ relative compactness
for a class of $2m$-component competition-diffusion systems and obtained
a Stefan limit in the four-component case
\cite{HilhorstEtAl2008RelativeCompactnessLp}.

Fast reaction limits have also been studied for a variety of other
reaction and diffusion structures. Bothe and Hilhorst considered a fast
reversible reaction system and obtained a nonlinear diffusion limit
\cite{BotheHilhorst2003ReactionDiffusionSystem}, while Bouillard, Eymard,
Henry, Herbin and Hilhorst studied fast precipitation and dissolution in
a porous medium
\cite{BouillardEtAl2009FastPrecipitationDissolution}.
Stephan established EDP-convergence for a linear reaction-diffusion
system with fast reversible reaction
\cite{Stephan2021EDPconvergenceLinearReactiondiffusion}.
Murakawa developed a general framework for singular limits of
reaction-diffusion systems
\cite{Murakawa2021FastReactionLimit}, and Crooks and Du studied
fast-reaction limits for systems with nonlinear diffusion
\cite{CrooksDu2024FastreactionLimitReaction}.
For a nonmonotone reaction function, Perthame and Skrzeczkowski
characterized oscillations of a nondiffusing component in terms of its
Young measure
\cite{PerthameSkrzeczkowski2023FastReactionLimit}.
In a related study of fast reaction and forward-backward diffusion,
Skrzeczkowski used a Radon--Nikodym approach to obtain a pointwise
characterization of the Young measure
\cite{Skrzeczkowski2022FastReactionLimit}.
These works involve reaction or diffusion structures different from the
monomial system considered below.

For monomial reaction terms that need not be proportional, Iida,
Monobe, Murakawa and Ninomiya
\cite{IidaEtAl2017VanishingMovingImmovable} considered
\[
\begin{cases}
\partial_tu_k=\Delta u_k-ku_k^{m_1}v_k^{m_2},\\
\partial_tv_k=-ku_k^{m_3}v_k^{m_4},
\end{cases}
\qquad m_1,m_2,m_3,m_4\ge1,
\]
with segregated initial data and homogeneous Neumann boundary conditions.
They studied representative one-parameter families in which three
exponents are fixed equal to one. Among the regimes treated there,
$(m,1,1,1)$ gives instantaneous interface disappearance for $m>3$;
$(1,m,1,1)$ with $m\ge1$ and $(1,1,1,m)$ with $1\le m<2$ give moving
interfaces governed by one-phase Stefan problems; and $(1,1,m,1)$ with
$m>1$ gives a stationary interface.
Hayashi derived vanishing, moving, and stationary interfaces directly from
two-species exclusion processes with unbalanced reaction rates in the
corresponding cases with integer exponents
\cite{Hayashi2021SpatialsegregationLimitExclusion}.

For the vanishing-interface family $(m,1,1,1)$, instantaneous interface
disappearance was extended from $m>3$ to $m>2$ in
\cite{Tsukamoto2025InterfaceDisappearanceFast}, and to every $m>1$ under
Dirichlet and Neumann boundary conditions in
\cite{Tsukamoto2026VanishingInterfacesAsymmetric}.
The stationary family $(1,1,m_3,m_4)$, with $m_3,m_4\ge1$ and either
$m_3>1$ or $m_4\ge2$, was studied under additional assumptions in
\cite{Tsukamoto2026FormationStationaryInterfaces}.

The representative regimes above, together with their related extensions,
have different interface dynamics but share a common feature. If
\[
\Omega_u(t):=\{x\in\Omega:u(x,t)>0\},
\]
then the limiting diffusive component satisfies
\[
\partial_tu=\Delta u
\qquad\text{in }\Omega_u(t).
\]
When the interface disappears, $\Omega_u(t)=\Omega$ for every $t>0$;
in the stationary case, $\Omega_u(t)$ is independent of $t$; and in the
Stefan case, $\Omega_u(t)$ evolves with time. In the latter case, the
effect of the reaction appears in the interfacial condition rather than
as a reaction term inside $\Omega_u(t)$. Thus, in these works, the main
difference lies in the evolution of the positive phase, while the equation
in its interior remains the heat equation
\cite{IidaEtAl2017VanishingMovingImmovable,
Tsukamoto2025InterfaceDisappearanceFast,
Tsukamoto2026FormationStationaryInterfaces,
Tsukamoto2026VanishingInterfacesAsymmetric}.

The present paper gives a different limiting equation within the same
four-exponent framework. We impose
\[
  m_2=m_4-1,
  \qquad
  m_1-m_3\ge1,
\]
and write
\[
  m_1=p+\alpha,\qquad m_2=q,\qquad
  m_3=p,\qquad m_4=q+1,
  \qquad p,q\ge1,\quad \alpha\ge1.
\]
For this family, we prove that the limiting diffusive component is positive
throughout the interior of the domain for every positive time. Nevertheless,
the fast reaction leaves a nonzero term in the limiting equation. 
More precisely, although $v_k\to0$ locally in
$\Omega\times(0,\infty)$,
\[
  kv_k^q\longrightarrow
  \ind{\{v_0>0\}}
  \left(q\int_0^t u(x,s)^p\,ds\right)^{-1}.
\]
Thus, although the nondiffusive component vanishes, the reaction term in
the $u_k$-equation does not vanish in the limit. The limiting equation for $u$ 
therefore contains a term involving the
time integral of $u^p$. This differs
from the representative families described above, where the limiting
equation in the positive phase is the heat equation.

Let \(N\in\Na\), let \(\Omega\subset\R^N\) be a bounded connected domain with
\(C^2\) boundary, and let
\[
  p,q\ge1,
  \qquad
  \alpha\ge1.
\]
Assume that
\begin{equation}\label{eq:initial-assumptions}
  u_0,v_0\in C(\overline\Omega),
  \qquad
  u_0\ge0,
  \quad
  v_0\ge0,
  \quad
  u_0v_0\equiv0,
  \quad
  u_0\not\equiv0,
  \quad
  v_0\not\equiv0.
\end{equation}
For each \(k>0\), let \((u_k,v_k)\) be the unique global nonnegative 
solution of
\begin{equation}\label{eq:Pk-intro}
\begin{cases}
  \partial_t u_k
  =\Delta u_k-k u_k^{p+\alpha}v_k^q
  &\text{in }\Omega\times(0,\infty),\\
  \partial_t v_k
  =-k u_k^p v_k^{q+1}
  &\text{in }\Omega\times(0,\infty),\\
  u_k(\cdot,0)=u_0,\qquad v_k(\cdot,0)=v_0
  &\text{in }\Omega.
\end{cases}
\end{equation}
For the diffusive component, fix one of the following boundary conditions and
use the same choice for every \(k>0\):
\begin{align}
  u_k=0
  &\quad\text{on }\partial\Omega\times(0,\infty)
  &&\text{(Dirichlet case)},
  \label{eq:dirichlet-bc}\\
  \partial_\nu u_k=0
  &\quad\text{on }\partial\Omega\times(0,\infty)
  &&\text{(Neumann case)}.
  \label{eq:neumann-bc}
\end{align}
Here \(\nu\) denotes the outward unit normal to \(\partial\Omega\).
Since the second equation contains no spatial derivatives, no boundary
condition is imposed on \(v_k\). In the Dirichlet case, assume
in addition that
\[
  u_0=0
  \qquad\text{on }\partial\Omega.
\]
Set
\begin{equation}\label{eq:D-definition}
D:= \{x\in\Omega:v_0(x)>0\}.
\end{equation}
Since $v_0$ is continuous, $D$ is open.

We now state the main result.

\begin{theorem}
\label{thm:memory-limit}
Under the above assumptions, the family $\{u_k\}_{k>0}$ is nonincreasing
with respect to $k$.
As $k\to\infty$, $u_k$ converges to a function
\[
u\in C(\overline\Omega\times(0,\infty))
\]
uniformly on $\overline\Omega\times[\tau,T]$ for every
$0<\tau<T<\infty$.
Moreover,
\[
u(x,t)>0
\qquad\text{for every }(x,t)\in\Omega\times(0,\infty).
\]

Define
\begin{equation}\label{eq:A-limit-definition}
A(x,t):=\int_0^t u(x,s)^p\,ds.
\end{equation}
Then $A(x,t)>0$ for every $x\in\Omega$ and $t>0$. Moreover,
\[
u\in W^{2,1}_{s,\mathrm{loc}}(\Omega\times(0,\infty))
\qquad(1\le s<\infty),
\]
and
\begin{equation}\label{eq:limit-equation}
\partial_tu
=
\Delta u-\ind{D}\frac{u^{p+\alpha}}{qA}
\quad\text{a.e. in }\Omega\times(0,\infty).
\end{equation}

Set
\begin{equation}\label{eq:Rk-intro}
R_k:=kv_k^q.
\end{equation}
Then, for every $1\le s<\infty$,
\begin{align}
R_k
&\longrightarrow R:=\ind{D}\frac1{qA}
&&\text{in }L^s_{\mathrm{loc}}(\Omega\times(0,\infty)),
\label{eq:Rk-limit-main}\\
k^{1/q}v_k=R_k^{1/q}
&\longrightarrow\ind{D}(qA)^{-1/q}
&&\text{in }L^s_{\mathrm{loc}}(\Omega\times(0,\infty)).
\label{eq:scaled-v-limit-main}
\end{align}
The convergence in \eqref{eq:Rk-limit-main} also holds almost everywhere.
In addition,
\begin{equation}\label{eq:limit-initial-trace}
u(\cdot,t)\longrightarrow u_0
\quad\text{in }L^2(\Omega)
\qquad\text{as }t\downarrow0.
\end{equation}

In the Dirichlet case,
\begin{equation}\label{eq:dirichlet-limit-boundary}
u=0\quad\text{on }\partial\Omega\times(0,\infty).
\end{equation}
In the Neumann case, for every $0<\tau<T<\infty$,
\begin{equation}\label{eq:neumann-c1-convergence}
u_k\longrightarrow u
\quad\text{in }C([\tau,T];C^1(\overline\Omega)),
\end{equation}
\begin{equation}\label{eq:neumann-limit-boundary}
\partial_\nu u=0\quad\text{on }\partial\Omega\times(0,\infty),
\end{equation}
and
\begin{equation}\label{eq:neumann-global-regularity}
u\in W^{2,1}_s(\Omega\times(\tau,T))
\qquad(1\le s<\infty).
\end{equation}
\end{theorem}
Theorem~\ref{thm:memory-limit} identifies the limit under the assumptions
above. If, in addition, the support of $u_0$ is compactly contained in
$\Omega$ and has positive distance from $D$, the limiting initial-boundary
value problem is also uniquely solvable.

\begin{theorem}\label{thm:memory-uniqueness}
Assume the hypotheses of Theorem~\ref{thm:memory-limit} and, in addition,
\begin{equation}\label{eq:separated-initial-supports}
\mathcal K:=\operatorname{supp}u_0\Subset\Omega,
\qquad
\operatorname{dist}(\mathcal K,D)>0.
\end{equation}
Then the limit $u$ in Theorem~\ref{thm:memory-limit} is the unique function $w$ 
satisfying the following conditions. Set
\[
A_w(x,t):=\int_0^t w(x,s)^p\,ds.
\]
For every $T>0$ and $0<\tau<T$,
\[
w\in
L^\infty(\Omega\times(0,T))
\cap C([0,T];L^2(\Omega))
\cap C(\overline\Omega\times(0,T]),
\]
\[
w>0
\qquad\text{in }\Omega\times(0,T],
\]
\[
w\in W^{2,1}_2(\Omega\times(\tau,T)),
\qquad
\ind{D}\frac{w^{p+\alpha}}{qA_w}
\in L^1(\Omega\times(0,T)),
\]
and
\[
\partial_tw
=
\Delta w-\ind{D}\frac{w^{p+\alpha}}{qA_w}
\quad\text{a.e. in }\Omega\times(0,T).
\]
Moreover,
\[
w(\cdot,t)\longrightarrow u_0
\quad\text{in }L^2(\Omega)
\qquad\text{as }t\downarrow0,
\]
and $w$ satisfies the corresponding homogeneous boundary condition.
In the Neumann case, we additionally require
\[
w\in C([\tau,T];C^1(\overline\Omega)).
\]
\end{theorem}

The additional assumption
\eqref{eq:separated-initial-supports} is used only for uniqueness and is
not needed for the convergence in Theorem~\ref{thm:memory-limit}.
No differentiability of $u_0$ or regularity of $\partial D$ is assumed.

Theorem~\ref{thm:memory-limit} separates the disappearance of the interface
from the disappearance of the reaction. The limit $u$ is positive at every
interior point for $t>0$, but on $D$ it satisfies an equation with the
nonzero reaction term $u^{p+\alpha}/(qA)$.
This term acts inside the positive phase, rather than only through an
interface condition. Thus $D$ remains explicitly in the limiting equation
even after the interface disappears.

The mechanism is visible in the exact formula for the second component.
With
\[
A_k(x,t)=\int_0^t u_k(x,s)^p\,ds,
\]
we have
\[
v_k(x,t)^q
=\frac{v_0(x)^q}{1+qkv_0(x)^qA_k(x,t)},
\qquad
R_k(x,t)
=\frac{kv_0(x)^q}{1+qkv_0(x)^qA_k(x,t)}.
\]
At points where $v_0(x)>0$, this gives
\[
R_k(x,t)
=\frac1{k^{-1}v_0(x)^{-q}+qA_k(x,t)}
\longrightarrow\frac1{qA(x,t)}.
\]
In particular, for every fixed positive time, $v_k$ has the scale
$k^{-1/q}$ at these points, so the factor $k$ compensates for the
smallness of $v_k^q$ in the first equation.
For comparison, in the vanishing-interface family
$(m_1,m_2,m_3,m_4)=(m,1,1,1)$, we have
\[
v_k(x,t)
=v_0(x)\exp\left(-k\int_0^t u_k(x,s)\,ds\right).
\]
A positive lower bound for the time integral then gives exponential decay in
$k$, and the reaction term vanishes locally at positive times in the
vanishing-interface regime
\cite{Tsukamoto2025InterfaceDisappearanceFast,
Tsukamoto2026VanishingInterfacesAsymmetric}.
In contrast, the algebraic formula in the present problem yields the
nonzero coefficient $\ind{D}/(qA)$ in \eqref{eq:limit-equation}.

The absence of diffusion in the second equation is essential here:
it allows the equation for $v_k$ to be solved explicitly at each spatial
point, and $v_k(x,t)=0$ whenever $v_0(x)=0$.
The time-integral term in the limiting equation follows directly from
this explicit formula.
Since the limiting coefficient $\ind{D}/(qA)$ may be discontinuous across
$\partial D$, the convergence of $R_k$ is stated in
$L^s_{\mathrm{loc}}$ for every finite $s$.
Under Dirichlet boundary conditions, the local nature of the convergence
for $v_k$ is also important: $v_k=v_0$ on $\partial\Omega$, so decay need
not occur at boundary points where $v_0>0$.

The main analytical step is to obtain a positive lower bound for $A_k$
uniformly in $k$. We adapt the one-dimensional and annular barriers used in
\cite{Tsukamoto2026VanishingInterfacesAsymmetric} to the reaction term that
survives here. The one-dimensional barrier is constructed to satisfy an
inequality involving its own time integral. For each $k>0$, we then define
a corresponding second component so that the resulting pair can be compared
directly with the solution of the original coupled system. An exponential
spatial weight controls the additional radial terms in the annular
construction.

The segregated initial data provide a ball on which $v_k=0$ and $u_k$
satisfies the heat equation. The resulting positive lower bound is then
propagated to any interior compact set through a finite chain of overlapping
balls. In the Neumann case, a boundary-distance argument extends this
estimate up to $\partial\Omega$.

These bounds control the reaction term and allow standard parabolic
regularity and compactness arguments.
Monotonicity in $k$ then gives convergence of the whole family.
Boundary comparison in the Dirichlet case and global parabolic estimates
in the Neumann case show that the boundary condition is inherited by the
limit.

The paper is organized as follows.
Section~\ref{sec:transform-comparison} establishes well-posedness and a
comparison principle for the general four-exponent system, together with
monotonicity in the reaction rate under $m_2\le m_4-1$.
Section~\ref{sec:one-dimensional-memory-barrier} derives the exact
representation of the nondiffusing component and constructs the
one-dimensional barrier.
Section~\ref{sec:annular-barriers-main-results} proves the uniform
estimates in the interior and near the boundary in the Neumann case and
completes the proof of Theorem~\ref{thm:memory-limit}.
Section~\ref{sec:memory-uniqueness} proves
Theorem~\ref{thm:memory-uniqueness}.

\section{Well-posedness, comparison, and monotonicity}
\label{sec:transform-comparison}
We use the standard solvability, the maximum principle, and parabolic
$L^r$-estimates for the linear heat equation with the prescribed
homogeneous boundary condition
\cite{LadyzhenskayaSolonnikovUraltseva1968LinearQuasilinearParabolic}.

We first establish well-posedness and a comparison principle
for the general four-exponent system
\eqref{eq:general-four-exponent-system}.
We then prove monotonicity with respect to the reaction rate
under the condition $m_2\le m_4-1$.

\begin{theorem}
\label{thm:general-wellposedness}
Let \(\Omega\subset\mathbb R^N\) be a bounded domain with \(C^2\) boundary,
and let
\[
  m_1,m_2,m_3,m_4\ge1.
\]
Assume that
\[
  u_0,v_0\in C(\overline\Omega),
  \qquad
  u_0\ge0,
  \qquad
  v_0\ge0.
\]
For each \(k>0\), consider
\begin{equation}\label{eq:general-four-exponent-system}
\begin{cases}
  \partial_tu_k
  =
  \Delta u_k-ku_k^{m_1}v_k^{m_2}
  &\text{in }\Omega\times(0,\infty),\\
  \partial_tv_k
  =
  -ku_k^{m_3}v_k^{m_4}
  &\text{in }\Omega\times(0,\infty),\\
  u_k(\cdot,0)=u_0,
  \qquad
  v_k(\cdot,0)=v_0
  &\text{in }\Omega.
\end{cases}
\end{equation}
Fix either the homogeneous Dirichlet condition
\begin{equation}\label{eq:general-dirichlet}
  u_k=0
  \qquad
  \text{on }\partial\Omega\times(0,\infty),
\end{equation}
or the homogeneous Neumann condition
\begin{equation}\label{eq:general-neumann}
  \partial_\nu u_k=0
  \qquad
  \text{on }\partial\Omega\times(0,\infty).
\end{equation}
In the Dirichlet case, assume in addition that
\[
  u_0=0
  \qquad
  \text{on }\partial\Omega.
\]

Then, for every $k>0$, problem
\eqref{eq:general-four-exponent-system}, with the chosen boundary
condition, has a unique global nonnegative solution $(u_k,v_k)$
satisfying
\[
\begin{aligned}
u_k
&\in
C\bigl([0,\infty);C(\overline\Omega)\bigr)
\cap
C^1\bigl((0,\infty);C(\overline\Omega)\bigr)
\cap
C\bigl((0,\infty);W^{2,r}(\Omega)\bigr),\\
v_k
&\in
C^1\bigl([0,\infty);C(\overline\Omega)\bigr)
\end{aligned}
\]
for every $N<r<\infty$.
Furthermore,
\begin{equation}\label{eq:general-basic-bounds}
\begin{aligned}
0\le u_k(x,t)
&\le \lVert u_0\rVert_{L^\infty(\Omega)},\\
0\le v_k(x,t)
&\le v_0(x)
\le \lVert v_0\rVert_{L^\infty(\Omega)}
\end{aligned}
\end{equation}
for every $(x,t)\in\overline\Omega\times[0,\infty)$.
In particular, for each $x\in\overline\Omega$, the function
$t\mapsto v_k(x,t)$ is nonincreasing.
\end{theorem}

\begin{proof}
Fix $k>0$. Let $\mathcal B w=w|_{\partial\Omega}$ in the
Dirichlet case and $\mathcal B w=\partial_\nu w|_{\partial\Omega}$
in the Neumann case, and set
\[
Q_{a,b}=\Omega\times(a,b).
\]

We use the standard solvability and maximum principle for the
linear heat equation with the prescribed homogeneous boundary
condition.
For $h>0$ and $F\in C(\overline\Omega\times[0,h])$, the problem
\begin{equation}\label{eq:linear-heat-problem}
\begin{cases}
\partial_tU-\Delta U=F
    & \text{in }Q_{0,h},\\
\mathcal B U=0
    & \text{on }\partial\Omega\times(0,h),\\
U(\cdot,0)=u_0
    & \text{in }\Omega
\end{cases}
\end{equation}
has a unique solution continuous on
$\overline\Omega\times[0,h]$, provided that
$u_0=0$ on $\partial\Omega$ in the Dirichlet case.
Moreover,
\begin{equation}\label{eq:linear-heat-sup-bound}
\lVert U\rVert_{C(\overline\Omega\times[0,h])}
\le
\lVert u_0\rVert_{L^\infty(\Omega)}
+
h\lVert F\rVert_{C(\overline\Omega\times[0,h])}.
\end{equation}

For $s\in\mathbb R$, write $s_+=\max\{s,0\}$, and define
\[
f_k(a,b)=-k(a_+)^{m_1}(b_+)^{m_2},
\qquad
g_k(a,b)=-k(a_+)^{m_3}(b_+)^{m_4}.
\]
Since $m_i\ge1$, the map $(f_k,g_k)$ is Lipschitz on bounded
subsets of $\mathbb R^2$.
For $h>0$, let $\mathcal Y_h=C(\overline\Omega\times[0,h])^2$,
and equip $\mathcal Y_h$ with the norm
\[
\|(\varphi,\psi)\|_h
=
\max\left\{
\|\varphi\|_{C(\overline\Omega\times[0,h])},
\|\psi\|_{C(\overline\Omega\times[0,h])}
\right\}.
\]
Given $(\varphi,\psi)\in\mathcal Y_h$, let $U$ be the solution
of \eqref{eq:linear-heat-problem} with
$F=f_k(\varphi,\psi)$, and define
\[
V(x,t)
=
v_0(x)
+
\int_0^t
g_k(\varphi(x,s),\psi(x,s))\,ds.
\]
Define $\Phi(\varphi,\psi)=(U,V)$.

Set
\[
R_0=\max\left\{
\lVert u_0\rVert_{L^\infty(\Omega)},
\lVert v_0\rVert_{L^\infty(\Omega)}
\right\},
\]
and fix $M>R_0$.
Let $B_M$ and $L_M$ be a bound and a Lipschitz constant for
$(f_k,g_k)$ on $[-M,M]^2$.
Estimate \eqref{eq:linear-heat-sup-bound}, applied also to
differences, and the integral formula for $V$ give
\[
\begin{aligned}
\lVert\Phi(\varphi_1,\psi_1)\rVert_h
&\le R_0+hB_M,\\
\lVert
\Phi(\varphi_1,\psi_1)-\Phi(\varphi_2,\psi_2)
\rVert_h
&\le
2hL_M
\lVert(\varphi_1-\varphi_2,\psi_1-\psi_2)\rVert_h
\end{aligned}
\]
whenever
$\lVert(\varphi_i,\psi_i)\rVert_h\le M$ for $i=1,2$.
Choose $h>0$ so that
\begin{equation}\label{eq:local-existence-time}
hB_M\le M-R_0,
\qquad
2hL_M<1.
\end{equation}
Then $\Phi$ is a contraction of the closed ball of radius $M$
in $\mathcal Y_h$ into itself.
Banach's fixed-point theorem therefore yields a unique local
solution of the system
\begin{equation}\label{eq:extended-four-exponent-system}
\begin{cases}
\partial_tu_k-\Delta u_k=f_k(u_k,v_k),\\
\partial_tv_k=g_k(u_k,v_k),
\end{cases}
\end{equation}
with the prescribed initial and boundary conditions.
Continuing the local construction yields a unique maximal solution
\[
(u_k,v_k)\in
C([0,T_{\max});C(\overline\Omega))^2,
\qquad
T_{\max}\in(0,\infty].
\]

The integral form of the second equation in
\eqref{eq:extended-four-exponent-system} gives
\begin{equation}\label{eq:vk-time-regularity}
v_k\in C^1([0,T_{\max});C(\overline\Omega)).
\end{equation}
Set $H_k=f_k(u_k,v_k)$.
Since $H_k$ is continuous, the parabolic $L^r$-estimates,
applied away from the initial time, yield
\begin{equation}\label{eq:uk-parabolic-regularity}
u_k\in W^{2,1}_r(Q_{a,b})
\qquad
(0<a<b<T_{\max},\ 1<r<\infty).
\end{equation}

Fix $1<r<\infty$ and $0<a<\tau<T<b<T_{\max}$.
By \eqref{eq:vk-time-regularity},
\eqref{eq:uk-parabolic-regularity},
and the local Lipschitz continuity of $f_k$,
\begin{equation}\label{eq:Hk-time-regularity}
H_k\in W^{1,r}(a,b;L^r(\Omega)).
\end{equation}
By \eqref{eq:Hk-time-regularity}, the linear $L^r$-estimate,
applied to the time difference quotients of the first equation
in \eqref{eq:extended-four-exponent-system} with a cutoff in time,
yields
\begin{equation}\label{eq:uk-time-parabolic-regularity}
\partial_tu_k\in W^{2,1}_r(Q_{\tau,T}).
\end{equation}

Taking $r>N+2$ in
\eqref{eq:uk-time-parabolic-regularity}, the parabolic Sobolev
embedding yields
\begin{equation}\label{eq:uk-time-regularity}
u_k\in C^1((0,T_{\max});C(\overline\Omega)).
\end{equation}
Since $\partial_tu_k-H_k=\Delta u_k$ is continuous,
elliptic regularity gives
\begin{equation}\label{eq:uk-spatial-regularity}
u_k\in C((0,T_{\max});W^{2,r}(\Omega))
\qquad (1<r<\infty).
\end{equation}

Let $(u_k)_-=\max\{-u_k,0\}$.
Taking $r>\max\{N,2\}$,
\eqref{eq:uk-time-regularity} and
\eqref{eq:uk-spatial-regularity} justify testing the first
equation in \eqref{eq:extended-four-exponent-system}
with $-(u_k)_-$ on $(\varepsilon,t)$, where
$0<\varepsilon<t<T_{\max}$.
The boundary term vanishes under either boundary condition,
and the reaction term vanishes because $(u_k)_+^{m_1}(u_k)_-=0$.
Hence
\begin{equation}\label{eq:uk-negative-part-estimate}
\frac12
\lVert (u_k)_-(t)\rVert_{L^2(\Omega)}^2
+
\int_\varepsilon^t
\lVert\nabla (u_k)_-(s)\rVert_{L^2(\Omega)}^2\,ds
=
\frac12
\lVert (u_k)_-(\varepsilon)\rVert_{L^2(\Omega)}^2.
\end{equation}
Letting $\varepsilon\downarrow0$ in
\eqref{eq:uk-negative-part-estimate}, continuity at the
initial time and $u_0\ge0$ give $u_k\ge0$.

For each fixed $x\in\overline\Omega$, the scalar equation for
$v_k(x,\cdot)$ admits the zero solution and has a right-hand
side locally Lipschitz in $v_k$.
Since $v_0\ge0$, uniqueness gives $v_k\ge0$.
The second equation in
\eqref{eq:extended-four-exponent-system} then yields
$\partial_t v_k\le0$, and hence
\begin{equation}\label{eq:vk-pointwise-bound}
0\le v_k(x,t)\le v_0(x)
\qquad
(x\in\overline\Omega,\ 0\le t<T_{\max}).
\end{equation}
Therefore the positive parts can be removed, and $(u_k,v_k)$
solves \eqref{eq:general-four-exponent-system}.

Since $\partial_tu_k-\Delta u_k\le0$, comparison with the
constant $\lVert u_0\rVert_{L^\infty(\Omega)}$ gives
\[
0\le u_k(x,t)\le\lVert u_0\rVert_{L^\infty(\Omega)}.
\]
This and \eqref{eq:vk-pointwise-bound} prove
\eqref{eq:general-basic-bounds} on
$\overline\Omega\times[0,T_{\max})$.
If $T_{\max}<\infty$, \eqref{eq:general-basic-bounds} shows that
\[
\max\left\{
\lVert u_k(t)\rVert_{L^\infty(\Omega)},
\lVert v_k(t)\rVert_{L^\infty(\Omega)}
\right\}
\le R_0
\qquad (0\le t<T_{\max}).
\]
Hence the same $M$, $B_M$, and $L_M$ can be used when the
solution is restarted at any $t_0<T_{\max}$.
By \eqref{eq:local-existence-time}, the solution can be extended
from any $t_0<T_{\max}$ by a time $h_*>0$ independent of $t_0$.
Taking $t_0$ sufficiently close to $T_{\max}$ contradicts
maximality.
Thus $T_{\max}=\infty$.
The theorem follows.
\end{proof}

We next prove a comparison principle for
\eqref{eq:general-four-exponent-system}.
The initial ordering is imposed in $L^2$ to allow for the
initial traces of the barriers constructed below.

\begin{theorem}\label{prop:coupled-comparison}
Let $G\subset\mathbb R^N$ be a bounded open set with $C^2$
boundary, let $t_0<t_1$, and set $Q=G\times(t_0,t_1)$.
Fix $k>0$ and $m_1,m_2,m_3,m_4\ge1$.
Let $\overline u,\overline v,\underline u,\underline v$
be nonnegative bounded functions on $Q$ such that, for every
$s\in(t_0,t_1)$,
\begin{equation}\label{eq:comparison-regularity}
\begin{aligned}
\overline u,\underline u
&\in W^{2,1}_2(G\times(s,t_1)),\\
\overline v,\underline v
&\in W^{1,2}(s,t_1;L^2(G)).
\end{aligned}
\end{equation}
Assume that
\begin{equation}\label{eq:comparison-upper-pair}
\begin{cases}
\partial_t\overline u
\ge
\Delta\overline u
-k\overline u^{m_1}\overline v^{m_2},\\
\partial_t\overline v
\le
-k\overline u^{m_3}\overline v^{m_4},
\end{cases}
\end{equation}
and
\begin{equation}\label{eq:comparison-lower-pair}
\begin{cases}
\partial_t\underline u
\le
\Delta\underline u
-k\underline u^{m_1}\underline v^{m_2},\\
\partial_t\underline v
\ge
-k\underline u^{m_3}\underline v^{m_4}
\end{cases}
\end{equation}
hold almost everywhere in $Q$.
On $\partial G\times(t_0,t_1)$, assume either
\begin{equation}\label{eq:comparison-boundary-order}
\underline u\le\overline u
\qquad\text{or}\qquad
\partial_\nu\underline u\le\partial_\nu\overline u,
\end{equation}
where $\nu$ denotes the outward unit normal
to $\partial G$.
If
\begin{equation}\label{eq:comparison-initial-order}
\lVert(\underline u(t)-\overline u(t))_+\rVert_{L^2(G)}
+
\lVert(\overline v(t)-\underline v(t))_+\rVert_{L^2(G)}
\longrightarrow0
\qquad(t\downarrow t_0),
\end{equation}
then
\begin{equation}\label{eq:comparison-conclusion}
\underline u\le\overline u,
\qquad
\overline v\le\underline v
\qquad\text{almost everywhere in }Q.
\end{equation}
\end{theorem}

\begin{proof}
Set
\[
U=(\underline u-\overline u)_+,
\qquad
V=(\overline v-\underline v)_+.
\]
Since the four functions are bounded on $Q$, choose $M\ge1$ such that
\[
0\le
\overline u,\underline u,\overline v,\underline v
\le M
\qquad\text{almost everywhere in }Q.
\]
Set
\[
C_1=km_2M^{m_1+m_2-1},
\qquad
C_2=km_3M^{m_3+m_4-1}.
\]

On $\{U>0\}$, we have $\underline u>\overline u$, so
\[
\begin{aligned}
&k\bigl(
\overline u^{m_1}\overline v^{m_2}
-\underline u^{m_1}\underline v^{m_2}
\bigr)\\
&\quad=
k\overline u^{m_1}
\bigl(\overline v^{m_2}-\underline v^{m_2}\bigr)
+
k\underline v^{m_2}
\bigl(\overline u^{m_1}-\underline u^{m_1}\bigr)\\
&\quad\le C_1V.
\end{aligned}
\]
Subtracting the first inequality in
\eqref{eq:comparison-upper-pair} from the first inequality in
\eqref{eq:comparison-lower-pair}, multiplying by $U$, and
integrating over $G$, we obtain
\begin{equation}\label{eq:comparison-u-energy}
\frac12\frac{d}{dt}\lVert U\rVert_{L^2(G)}^2
+
\lVert\nabla U\rVert_{L^2(G)}^2
\le
C_1\int_G UV\,dx.
\end{equation}
Here the boundary contribution is nonpositive by
\eqref{eq:comparison-boundary-order}.

Similarly, on $\{V>0\}$, we have $\overline v>\underline v$,
and hence
\[
\begin{aligned}
&k\bigl(
\underline u^{m_3}\underline v^{m_4}
-\overline u^{m_3}\overline v^{m_4}
\bigr)\\
&\quad=
k\underline v^{m_4}
\bigl(\underline u^{m_3}-\overline u^{m_3}\bigr)
+
k\overline u^{m_3}
\bigl(\underline v^{m_4}-\overline v^{m_4}\bigr)\\
&\quad\le C_2U.
\end{aligned}
\]
Subtracting the second inequality in
\eqref{eq:comparison-lower-pair} from the second inequality in
\eqref{eq:comparison-upper-pair}, multiplying by $V$, and
integrating over $G$, we obtain
\begin{equation}\label{eq:comparison-v-energy}
\frac12\frac{d}{dt}\lVert V\rVert_{L^2(G)}^2
\le
C_2\int_G UV\,dx.
\end{equation}

Set
\[
E(t)=
\lVert U(t)\rVert_{L^2(G)}^2
+
\lVert V(t)\rVert_{L^2(G)}^2.
\]
Adding \eqref{eq:comparison-u-energy} and
\eqref{eq:comparison-v-energy}, and using $2UV\le U^2+V^2$,
we obtain
\[
E'(t)\le(C_1+C_2)E(t)
\qquad\text{for almost every }t\in(t_0,t_1).
\]
Thus, for $t_0<s<t<t_1$, Gronwall's inequality gives
\[
E(t)\le e^{(C_1+C_2)(t-s)}E(s).
\]
Letting $s\downarrow t_0$ and using
\eqref{eq:comparison-initial-order} yields $E(t)=0$.
The theorem follows.
\end{proof}

We next use the comparison principle to compare solutions
corresponding to different reaction rates.

\begin{lemma}\label{lem:reaction-rate-monotonicity}
Assume the hypotheses of
Theorem~\ref{thm:general-wellposedness}, and suppose that
\begin{equation}\label{eq:monotonicity-exponent-condition}
m_2\le m_4-1.
\end{equation}
For each $k>0$, let $(u_k,v_k)$ be the solution of
\eqref{eq:general-four-exponent-system} with the prescribed
initial and boundary conditions.
Then, for every $0<k_1<k_2$,
\begin{equation}\label{eq:reaction-rate-monotonicity}
u_{k_2}\le u_{k_1},
\qquad
k_1^{1/m_2}v_{k_1}\le k_2^{1/m_2}v_{k_2}
\end{equation}
on $\overline\Omega\times[0,\infty)$.
\end{lemma}

\begin{proof}
Fix $0<k_1<k_2$ and $T>0$, and set
\[
\lambda=\left(\frac{k_2}{k_1}\right)^{1/m_2},
\qquad
\widetilde v=\lambda v_{k_2}.
\]
Since $\lambda^{m_2}=k_2/k_1$, we have
\begin{equation}\label{eq:monotonicity-rescaled-system}
\begin{aligned}
\partial_tu_{k_2}
&=
\Delta u_{k_2}
-k_1u_{k_2}^{m_1}\widetilde v^{m_2},\\
\partial_t\widetilde v
&=
-k_1\lambda^{m_2-m_4+1}
u_{k_2}^{m_3}\widetilde v^{m_4}\\
&\ge
-k_1u_{k_2}^{m_3}\widetilde v^{m_4},
\end{aligned}
\end{equation}
where the last inequality follows from $\lambda>1$ and
\eqref{eq:monotonicity-exponent-condition}.
Moreover,
\[
u_{k_2}(\cdot,0)=u_{k_1}(\cdot,0)=u_0,
\qquad
v_{k_1}(\cdot,0)=v_0
\le\lambda v_0=\widetilde v(\cdot,0),
\]
and $u_{k_1}$ and $u_{k_2}$ satisfy the same homogeneous
boundary condition.
By \eqref{eq:monotonicity-rescaled-system}, we may apply
Theorem~\ref{prop:coupled-comparison} on $\Omega\times(0,T)$
with reaction coefficient $k_1$ and
\[
(\overline u,\overline v)=(u_{k_1},v_{k_1}),
\qquad
(\underline u,\underline v)=(u_{k_2},\widetilde v).
\]
We obtain
\[
u_{k_2}\le u_{k_1},
\qquad
v_{k_1}\le\widetilde v
\]
almost everywhere in $\Omega\times(0,T)$.
By Theorem~\ref{thm:general-wellposedness}, these functions are
continuous on $\overline\Omega\times[0,T]$, and hence the
inequalities hold throughout $\overline\Omega\times[0,T]$.
Since $k_1^{1/m_2}\widetilde v=k_2^{1/m_2}v_{k_2}$
and $T>0$ is arbitrary, the lemma follows.
\end{proof}

\section{A one-dimensional memory barrier}
\label{sec:one-dimensional-memory-barrier}

We now take
\[
m_1=p+\alpha,\qquad m_2=q,\qquad
m_3=p,\qquad m_4=q+1,
\qquad p,q\ge1,\quad \alpha\ge1.
\]
Set
\begin{equation}\label{eq:memory-variables}
R_k(x,t):=kv_k(x,t)^q,
\qquad
A_k(x,t):=\int_0^t u_k(x,s)^p\,ds.
\end{equation}
Then \eqref{eq:general-four-exponent-system} gives
\begin{equation}\label{eq:transformed-critical-system}
\begin{cases}
\partial_tu_k=\Delta u_k-R_ku_k^{p+\alpha},\\
\partial_tR_k=-qu_k^pR_k^2,\\
u_k(\cdot,0)=u_0,\qquad R_k(\cdot,0)=kv_0^q,
\end{cases}
\end{equation}
with the same boundary condition for $u_k$.

Solving the second equation in
\eqref{eq:transformed-critical-system}, we obtain
\begin{equation}\label{eq:exact-memory-representation}
\begin{aligned}
R_k(x,t)
&=
\frac{kv_0(x)^q}
{1+qkv_0(x)^qA_k(x,t)},\\
v_k(x,t)^q
&=
\frac{v_0(x)^q}
{1+qkv_0(x)^qA_k(x,t)}.
\end{aligned}
\end{equation}
These identities also hold when $v_0(x)=0$, in which case
$v_k(x,t)=R_k(x,t)=0$.
At each point where $v_0(x)>0$,
\eqref{eq:exact-memory-representation} can also be written as
\begin{equation}\label{eq:reciprocal-memory-representation}
R_k(x,t)
=
\frac{1}{k^{-1}v_0(x)^{-q}+qA_k(x,t)}.
\end{equation}

Since $m_2=m_4-1$,
Lemma~\ref{lem:reaction-rate-monotonicity} yields
\begin{equation}\label{eq:critical-reaction-rate-monotonicity}
u_{k_2}\le u_{k_1},
\qquad
R_{k_1}\le R_{k_2},
\qquad 0<k_1<k_2,
\end{equation}
on $\overline\Omega\times[0,\infty)$.

We next construct a one-dimensional barrier adapted to the time-integral term 
in the transformed equation.

\begin{lemma}\label{lem:one-dimensional-memory-barrier}
Let $p,q\ge1$ and $\alpha\ge1$, and set
\begin{equation}\label{eq:memory-barrier-constants}
a_w:=1+\frac{4pe^{1/4}}{q},
\qquad
c_w:=\sqrt{2a_w+\frac{4e^{1/4}}{q}+1}.
\end{equation}
For $L>0$ and $0<B\le1$, define
\begin{equation}\label{eq:one-dimensional-memory-barrier}
W(y,t)
=
B\left(1-\frac{y}{L}\right)^2
\exp\left(-\frac{c_w y}{\sqrt t}-\frac{a_w y^2}{t}\right)
\end{equation}
for $0\le y\le L$ and $t>0$, and set
\[
\mathcal J(y,t):=\int_0^t W(y,s)^p\,ds.
\]
Then
\begin{equation}\label{eq:one-dimensional-barrier-inequality}
\partial_tW
\le
\partial_{yy}W-\frac{W^{p+\alpha}}{q\mathcal J}
\quad\text{in }(0,L)\times(0,\infty).
\end{equation}
\end{lemma}

\begin{proof}
Set $C_0=4e^{1/4}$. For $\eta\ge0$, define
\[
F(\eta)=e^{-c_w\eta-a_w\eta^2},
\qquad
G(\eta)=\int_0^1
F\left(\frac{\eta}{\sqrt s}\right)^p\,ds.
\]
We first estimate $F^p/G$.
Fix $\eta\ge0$, and set
\[
\ell=1+pc_w\eta+2pa_w\eta^2,
\qquad
\delta=\frac{1}{4\ell}.
\]
Since $0<\delta\le1/4$, every $s\in[1-\delta,1]$
satisfies $s\ge3/4$ and
\[
s^{-1/2}-1
\le1-s,
\qquad
s^{-1}-1\le2(1-s).
\]
Consequently,
\[
pc_w\eta(s^{-1/2}-1)+pa_w\eta^2(s^{-1}-1)
\le
(pc_w\eta+2pa_w\eta^2)\delta
\le\frac14,
\]
and hence
\[
F\left(\frac{\eta}{\sqrt s}\right)^p
\ge e^{-1/4}F(\eta)^p
\qquad (1-\delta\le s\le1).
\]
Integrating over $[1-\delta,1]$ gives
\[
G(\eta)\ge\delta e^{-1/4}F(\eta)^p.
\]
Thus
\begin{equation}\label{eq:memory-profile-quotient-bound}
\frac{F(\eta)^p}{G(\eta)}
\le
C_0(1+pc_w\eta+2pa_w\eta^2)
\qquad (\eta\ge0).
\end{equation}

Since
\[
F'=-(c_w+2a_w\eta)F,
\qquad
F''=\bigl((c_w+2a_w\eta)^2-2a_w\bigr)F,
\]
\eqref{eq:memory-barrier-constants} and
\eqref{eq:memory-profile-quotient-bound} yield
\begin{equation}\label{eq:memory-profile-inequality}
\begin{aligned}
F''+\frac{\eta}{2}F'
-\frac{F^{p+1}}{qG}
&\ge
\left[
1+c_w\left(3a_w+\frac12\right)\eta
+a_w(2a_w+1)\eta^2
\right]F\\
&\ge0
\qquad (\eta\ge0).
\end{aligned}
\end{equation}

Now set
\[
\chi(y)=\left(1-\frac{y}{L}\right)^2,
\qquad
\eta=\frac{y}{\sqrt t}.
\]
By \eqref{eq:one-dimensional-memory-barrier},
\[
W(y,t)=B\chi(y)F(\eta),
\qquad
\mathcal J(y,t)=tB^p\chi(y)^pG(\eta).
\]
Therefore, in $(0,L)\times(0,\infty)$,
\[
\begin{aligned}
\partial_{yy}W-\partial_tW
-\frac{W^{p+1}}{q\mathcal J}
&=
\frac{B\chi}{t}
\left(
F''+\frac{\eta}{2}F'
-\frac{F^{p+1}}{qG}
\right)\\
&\quad
+\frac{2B\chi'}{\sqrt t}F'
+B\chi''F
\ge0,
\end{aligned}
\]
where the last inequality follows from
\eqref{eq:memory-profile-inequality},
$\chi'\le0$, $\chi''\ge0$, and $F'\le0$.
Since $0<W\le B\le1$ and $\alpha\ge1$,
we have $W^{p+\alpha}\le W^{p+1}$, which proves
\eqref{eq:one-dimensional-barrier-inequality}.
The lemma follows.
\end{proof}

To use Theorem~\ref{prop:coupled-comparison}, we introduce
a second component associated with $W$.

\begin{lemma}\label{lem:regularized-barrier-system}
Let $W$ and $\mathcal J$ be as in
Lemma~\ref{lem:one-dimensional-memory-barrier}.
For $M>0$ and $k>0$, define
\begin{equation}\label{eq:regularized-barrier-component}
Z_k(y,t)
:=
\frac{M}{\left(1+qkM^q\mathcal J(y,t)\right)^{1/q}},
\qquad 0\le y\le L,\quad t>0.
\end{equation}
Then
\begin{equation}\label{eq:regularized-barrier-system}
\begin{cases}
\partial_tW
\le
\partial_{yy}W-kW^{p+\alpha}Z_k^q,\\
\partial_tZ_k
=
-kW^pZ_k^{q+1}
\end{cases}
\quad\text{in }(0,L)\times(0,\infty).
\end{equation}
\end{lemma}

\begin{proof}
By \eqref{eq:regularized-barrier-component},
\[
kZ_k^q
=
\frac{1}{(kM^q)^{-1}+q\mathcal J}
\le
\frac{1}{q\mathcal J}
\quad\text{in }(0,L)\times(0,\infty).
\]
Hence \eqref{eq:one-dimensional-barrier-inequality} gives
\[
\partial_tW
\le
\partial_{yy}W-\frac{W^{p+\alpha}}{q\mathcal J}
\le
\partial_{yy}W-kW^{p+\alpha}Z_k^q.
\]
Since $\partial_t\mathcal J=W^p$, differentiating
\eqref{eq:regularized-barrier-component} yields
\[
\partial_tZ_k
=
-kM^{q+1}W^p
\left(1+qkM^q\mathcal J\right)^{-1-1/q}
=
-kW^pZ_k^{q+1}.
\]
The lemma follows.
\end{proof}

\section{Weighted annular barriers and proof of the main theorem}
\label{sec:annular-barriers-main-results}

We extend the one-dimensional barrier to annuli and use it to
propagate a positive lower bound through $\Omega$.
We then prove Theorems~\ref{thm:positivity-decay} and
\ref{thm:memory-limit}.

For $x_0\in\mathbb R^N$ and $0<d_1<d_2$, set
\[
\mathcal A_{d_1,d_2}(x_0)
:=
\{x\in\mathbb R^N:d_1<|x-x_0|<d_2\}.
\]

\begin{lemma}\label{lem:weighted-annular-barrier}
Let $\mathcal A=\mathcal A_{d_1,d_2}(x_0)$, and choose
\[
\sigma\ge\frac{N-1}{d_1}.
\]
Let $W$ and $\mathcal J$ be as in
Lemma~\ref{lem:one-dimensional-memory-barrier}, with
$L=d_2-d_1$ and $0<B\le1$.
For $x\in\overline{\mathcal A}$, set
\[
r=|x-x_0|,
\qquad
h(r)=e^{-\sigma(r-d_1)}.
\]
For $M>0$ and $k>0$, define
\begin{equation}\label{eq:weighted-annular-barrier}
\begin{aligned}
U(x,t)&:=h(r)W(r-d_1,t),\\
\mathcal I(x,t)&:=\int_0^t U(x,s)^p\,ds,\\
V_k(x,t)&:=
\frac{M}{\left(1+qkM^q\mathcal I(x,t)\right)^{1/q}},
\end{aligned}
\qquad x\in\overline{\mathcal A},\quad t>0.
\end{equation}
Then
\begin{equation}\label{eq:weighted-annular-system}
\begin{cases}
\partial_tU\le\Delta U-kU^{p+\alpha}V_k^q,\\
\partial_tV_k=-kU^pV_k^{q+1}
\end{cases}
\quad\text{in }\mathcal A\times(0,\infty).
\end{equation}
\end{lemma}

\begin{proof}
In the following computations, $W$, $\mathcal J$, and
the derivatives of $W$ are evaluated at $(r-d_1,t)$.
Since $U$ is radial about $x_0$,
\[
\begin{aligned}
\Delta U-\partial_tU
=h(r)\Bigg[
&\partial_{yy}W-\partial_tW
+\left(\frac{N-1}{r}-2\sigma\right)\partial_yW\\
&+\sigma\left(\sigma-\frac{N-1}{r}\right)W
\Bigg].
\end{aligned}
\]
By \eqref{eq:one-dimensional-memory-barrier},
$\partial_yW\le0$ and $W\ge0$.
Since $0\le(N-1)/r\le(N-1)/d_1\le\sigma$, we obtain
\[
\Delta U-\partial_tU
\ge h(r)(\partial_{yy}W-\partial_tW).
\]
Moreover, $\mathcal I=h(r)^p\mathcal J$, $0<h(r)\le1$,
and $\alpha\ge1$.
Thus \eqref{eq:one-dimensional-barrier-inequality} gives
\[
\Delta U-\partial_tU
\ge\frac{h(r) W^{p+\alpha}}{q\mathcal J}
\ge\frac{h(r)^\alpha W^{p+\alpha}}{q\mathcal J}
=\frac{U^{p+\alpha}}{q\mathcal I}.
\]

By \eqref{eq:weighted-annular-barrier},
\[
kV_k^q
=
\frac{kM^q}{1+qkM^q\mathcal I}
=
\frac{1}{(kM^q)^{-1}+q\mathcal I}
\le
\frac{1}{q\mathcal I}.
\]
Hence
\[
\partial_tU
\le
\Delta U-\frac{U^{p+\alpha}}{q\mathcal I}
\le
\Delta U-kU^{p+\alpha}V_k^q.
\]
Since $\partial_t\mathcal I=U^p$, differentiating
\eqref{eq:weighted-annular-barrier} yields
\[
\begin{aligned}
\partial_tV_k
&=
-\frac{M}{q}
\left(1+qkM^q\mathcal I\right)^{-1-1/q}
qkM^qU^p\\
&=
-kU^pV_k^{q+1}.
\end{aligned}
\]
These relations prove \eqref{eq:weighted-annular-system}.
The lemma follows.
\end{proof}

We next use Lemma~\ref{lem:weighted-annular-barrier} to propagate
a positive lower bound from one ball to a nearby ball.

\begin{lemma}\label{lem:ball-to-ball-propagation}
Assume the hypotheses of Theorem~\ref{thm:positivity-decay}.
Let $R>0$ and $x_0,x_1\in\Omega$ satisfy
\begin{equation}\label{eq:propagation-geometry}
\overline{B_{3R/2}(x_0)}\subset\Omega,
\qquad
|x_1-x_0|\le\frac R4.
\end{equation}
Suppose that, for some $B_0>0$ and $0\le t_0<T$,
\begin{equation}\label{eq:propagation-source-bound}
u_k\ge B_0
\quad\text{on }\overline{B_R(x_0)}\times[t_0,T]
\end{equation}
for every $k>0$.
Then, for each $0<\theta<T-t_0$, there exists $B_1>0$,
independent of $k$, such that
\begin{equation}\label{eq:propagation-target-bound}
u_k\ge B_1
\quad\text{on }\overline{B_R(x_1)}\times[t_0+\theta,T]
\end{equation}
for every $k>0$.
\end{lemma}

\begin{proof}
Set
\[
\mathcal A=\mathcal A_{R,3R/2}(x_0),
\qquad
B=\min\{1,B_0\},
\qquad
M=\|v_0\|_{L^\infty(\Omega)}.
\]
Fix $k>0$, and let $U$, $\mathcal I$, and $V_k$ be given by
Lemma~\ref{lem:weighted-annular-barrier}, with
\[
d_1=R,\qquad d_2=\frac{3R}{2},
\qquad \sigma=\frac{N-1}{R}.
\]
On $Q=\mathcal A\times(t_0,T)$, define
\[
\underline u(x,t)=U(x,t-t_0),
\qquad
\underline v(x,t)=V_k(x,t-t_0).
\]
By \eqref{eq:weighted-annular-system}, this pair satisfies
\eqref{eq:comparison-lower-pair}, with
$m_1=p+\alpha$, $m_2=q$, $m_3=p$, and $m_4=q+1$.

We verify the remaining hypotheses of
Theorem~\ref{prop:coupled-comparison}.
By \eqref{eq:weighted-annular-barrier},
\[
0\le\underline u\le B,
\qquad
0<\underline v\le M.
\]
Together with \eqref{eq:general-basic-bounds}, these bounds show
that $\underline u$, $\underline v$, $u_k$, and $v_k$
are nonnegative and bounded on $Q$.
By Theorem~\ref{thm:general-wellposedness} and
\eqref{eq:weighted-annular-barrier}, we also have
\[
\begin{aligned}
u_k,\underline u
&\in W^{2,1}_2(\mathcal A\times(s,T)),\\
v_k,\underline v
&\in W^{1,2}(s,T;L^2(\mathcal A))
\end{aligned}
\qquad (t_0<s<T).
\]
On the inner and outer boundaries, respectively,
\[
\begin{aligned}
\underline u=B\le u_k
&\quad\text{when }|x-x_0|=R,\\
\underline u=0\le u_k
&\quad\text{when }|x-x_0|=\frac{3R}{2},
\end{aligned}
\qquad t_0<t<T,
\]
by \eqref{eq:propagation-source-bound} and nonnegativity.

Since $u_k\ge0$ and $v_k\le M$,
\eqref{eq:weighted-annular-barrier} and dominated convergence give
\[
\|(\underline u(t)-u_k(t))_+\|_{L^2(\mathcal A)}
+
\|(v_k(t)-\underline v(t))_+\|_{L^2(\mathcal A)}
\longrightarrow0
\qquad(t\downarrow t_0).
\]
We may therefore apply Theorem~\ref{prop:coupled-comparison}
with $(\overline u,\overline v)=(u_k,v_k)$.
By continuity, its conclusion gives
\begin{equation}\label{eq:annular-comparison}
u_k(x,t)\ge U(x,t-t_0),
\qquad
v_k(x,t)\le V_k(x,t-t_0)
\end{equation}
on $\overline{\mathcal A}\times(t_0,T]$.

For $R\le|x-x_0|\le5R/4$ and $t_0+\theta\le t\le T$,
\eqref{eq:one-dimensional-memory-barrier} and
\eqref{eq:weighted-annular-barrier} give
\[
U(x,t-t_0)
\ge
B_1:=
\frac B4
\exp\left(
-\frac{N-1}{4}
-\frac{c_wR}{4\sqrt\theta}
-\frac{a_wR^2}{16\theta}
\right)>0.
\]
Since $B_1\le B_0$ and
\[
\overline{B_R(x_1)}
\subset
\overline{B_{5R/4}(x_0)},
\]
\eqref{eq:propagation-source-bound} and
\eqref{eq:annular-comparison} prove
\eqref{eq:propagation-target-bound}.
The lemma follows.
\end{proof}

Applying Lemma~\ref{lem:ball-to-ball-propagation} along a finite chain
of balls yields the following uniform estimates.

\begin{theorem}\label{thm:positivity-decay}
Assume the hypotheses stated in Section~\ref{sec:introduction}.
Then, for every compact set $K\Subset\Omega$ and every
$0<\tau<T<\infty$, there exist constants
\[
\rho_{K,\tau,T}>0,
\qquad
C_{K,\tau,T}>0,
\]
independent of $k$, such that
\begin{align}
u_k(x,t)&\ge\rho_{K,\tau,T},
&&(x,t)\in K\times[\tau,T],
\label{eq:main-positive-lower-bound}\\
0\le v_k(x,t)&\le C_{K,\tau,T}k^{-1/q},
&&(x,t)\in K\times[\tau,T].
\label{eq:main-v-decay}
\end{align}
In particular, for every $x\in\Omega$ and $t>0$,
\[
\liminf_{k\to\infty}u_k(x,t)>0,
\qquad
v_k(x,t)\longrightarrow0.
\]
\end{theorem}

\begin{proof}
Fix a compact set $K\Subset\Omega$ and $0<\tau<T<\infty$.
By \eqref{eq:initial-assumptions}, choose $x_*\in\Omega$
and $r_*>0$ such that
\[
\overline{B_{4r_*}(x_*)}\subset\Omega,
\qquad
\min_{\overline{B_{4r_*}(x_*)}}u_0>0.
\]
Since $u_0v_0=0$, we have $v_0=0$ on this ball.
By \eqref{eq:general-basic-bounds}, $v_k=0$ there
for every $k>0$ and $t\ge0$.
Thus \eqref{eq:Pk-intro} gives
\[
\partial_tu_k=\Delta u_k
\quad\text{in }B_{4r_*}(x_*)\times(0,\infty).
\]

Choose a nonnegative function
$\varphi\in C_c^\infty(B_{4r_*}(x_*))$ such that
\[
\varphi\le u_0,
\qquad
\min_{\overline{B_{2r_*}(x_*)}}\varphi>0.
\]
Let $H$ solve
\[
\begin{cases}
\partial_tH=\Delta H
&\text{in }B_{4r_*}(x_*)\times(0,\infty),\\
H=0
&\text{on }\partial B_{4r_*}(x_*)\times(0,\infty),\\
H(\cdot,0)=\varphi
&\text{in }B_{4r_*}(x_*).
\end{cases}
\]
The comparison principle gives $H\le u_k$ for every $k>0$.
By the strong maximum principle, continuity, and the choice
of $\varphi$,
\[
B_*:=
\min_{\overline{B_{2r_*}(x_*)}\times[0,T]}H>0.
\]
Consequently,
\begin{equation}\label{eq:uniform-source-ball}
u_k\ge B_*
\quad\text{on }\overline{B_{2r_*}(x_*)}\times[0,T]
\end{equation}
for every $k>0$.

Fix $x\in K$.
Since $\Omega$ is connected, there exists a polygonal path
$\gamma:[0,1]\to\Omega$ joining $x_*$ to $x$.
Set
\[
\delta_x:=
\operatorname{dist}(\gamma([0,1]),\partial\Omega)>0,
\]
and choose
\[
0<R_x<\min\left\{2r_*,\frac{\delta_x}{2}\right\}.
\]
By uniform continuity of $\gamma$, choose finitely many points
$x_0,\ldots,x_m$ on its image such that
\[
x_0=x_*,
\qquad
x_m=x,
\qquad
|x_j-x_{j-1}|\le\frac{R_x}{4}
\quad(1\le j\le m).
\]
The choice of $R_x$ also gives
\[
\overline{B_{3R_x/2}(x_j)}\subset\Omega
\qquad(0\le j\le m).
\]
Set
\[
\theta_x=\frac{\tau}{m+1},
\qquad
t_j=j\theta_x
\quad(0\le j\le m).
\]
Starting from \eqref{eq:uniform-source-ball}, apply
Lemma~\ref{lem:ball-to-ball-propagation} successively along
this chain.
There exist constants $B_j>0$, independent of $k$, with
$B_0=B_*$, such that
\[
u_k\ge B_j
\quad\text{on }\overline{B_{R_x}(x_j)}\times[t_j,T]
\qquad(0\le j\le m)
\]
for every $k>0$.
Since $t_m<\tau$, we obtain
\[
u_k\ge B_m
\quad\text{on }\overline{B_{R_x}(x)}\times[\tau,T]
\]
for every $k>0$.
Taking a finite subcover of $K$ by the balls $B_{R_x}(x)$
and the minimum of the corresponding lower bounds proves
\eqref{eq:main-positive-lower-bound}.

Applying \eqref{eq:main-positive-lower-bound} with $\tau/2$
in place of $\tau$, choose $\rho>0$, independent of $k$, such that
\[
u_k\ge\rho
\quad\text{on }K\times[\tau/2,T].
\]
For $(x,t)\in K\times[\tau,T]$, we have
\[
A_k(x,t)
\ge\int_{\tau/2}^{\tau}u_k(x,s)^p\,ds
\ge\frac{\tau\rho^p}{2}.
\]
Hence \eqref{eq:exact-memory-representation} gives
\begin{equation}\label{eq:positive-time-reaction-bound}
0\le R_k(x,t)
=
\frac{kv_0(x)^q}{1+qkv_0(x)^qA_k(x,t)}
\le\frac{1}{qA_k(x,t)}
\le\frac{2}{q\tau\rho^p}.
\end{equation}
Since $v_k=k^{-1/q}R_k^{1/q}$, we conclude that
\[
0\le v_k(x,t)
\le
\left(\frac{2}{q\tau\rho^p}\right)^{1/q}k^{-1/q}
\quad\text{on }K\times[\tau,T],
\]
which proves \eqref{eq:main-v-decay}.
The theorem follows.
\end{proof}

We next extend the positive lower bound to the boundary in the
Neumann case.

\begin{lemma}\label{lem:neumann-boundary-estimates}
Assume the hypotheses stated in Section~\ref{sec:introduction}, with
\eqref{eq:neumann-bc}.
For every $0<\tau<T<\infty$, there exist constants
$\rho_{\tau,T},C_{\tau,T}>0$, independent of $k$, such that
\begin{equation}\label{eq:neumann-global-bounds}
u_k\ge\rho_{\tau,T},\qquad 0\le R_k\le C_{\tau,T}
\quad\text{on }\overline\Omega\times[\tau,T]
\end{equation}
for every $k>0$.
\end{lemma}

\begin{proof}
Write $d(x)=\operatorname{dist}(x,\partial\Omega)$.
Since $\partial\Omega$ is $C^2$, choose $\delta>0$ small enough that
$d$ is $C^2$ up to the boundary in
\[
\mathcal C_\delta=\{x\in\Omega:0<d(x)<\delta\},
\]
with $|\nabla d|=1$ in $\mathcal C_\delta$,
$\partial_\nu d=-1$ on $\partial\Omega$, and
$\Delta d$ bounded in $\mathcal C_\delta$.
Choose $\delta$ also so that
$K_\delta:=\{x\in\overline\Omega:d(x)\ge\delta\}$ is nonempty.
Then $K_\delta\Subset\Omega$, and
\[
\partial\mathcal C_\delta=\partial\Omega\cup\Gamma_\delta,
\qquad \Gamma_\delta=\{x\in\Omega:d(x)=\delta\}.
\]

Fix $0<\tau<T<\infty$ and set $t_0=\tau/2$.
By \eqref{eq:main-positive-lower-bound}, choose $B_0>0$, independent
of $k$, such that $u_k\ge B_0$ on $K_\delta\times[t_0,T]$.
Set
\[
B=\min\{1,B_0\},\qquad M=\|v_0\|_{L^\infty(\Omega)},\qquad
\sigma=\|\Delta d\|_{L^\infty(\mathcal C_\delta)}.
\]
Let $W$ and $\mathcal J$ be as in
Lemma~\ref{lem:one-dimensional-memory-barrier}, with $L=2\delta$.
For $x\in\overline{\mathcal C_\delta}$ and $s>0$, define
\[
\begin{aligned}
y(x)&=\delta-d(x),\\
U(x,s)&=e^{-\sigma y(x)}W(y(x),s),\\
\mathcal I(x,s)&=\int_0^s U(x,r)^p\,dr,\\
V_k(x,s)&=\frac{M}{(1+qkM^q\mathcal I(x,s))^{1/q}}.
\end{aligned}
\]

To simplify notation, we write
\[
W=W(y(x),s),\qquad
\mathcal J=\mathcal J(y(x),s),
\]
and use the same convention for the derivatives of $W$.
Since $|\nabla y|=1$, we have
\[
\begin{aligned}
\Delta U-\partial_sU
=e^{-\sigma y}\Big[
&\partial_{yy}W-\partial_sW
+(\Delta y-2\sigma)\partial_yW\\
&+\sigma(\sigma-\Delta y)W
\Big].
\end{aligned}
\]
Since $\Delta y\le\sigma$, $\sigma\ge0$,
$\partial_yW\le0$, and $W\ge0$, we obtain
\[
\Delta U-\partial_sU
\ge e^{-\sigma y}(\partial_{yy}W-\partial_sW).
\]
Moreover, $\mathcal I=e^{-p\sigma y}\mathcal J$.
Thus \eqref{eq:one-dimensional-barrier-inequality} and $\alpha\ge1$ give
\[
\Delta U-\partial_sU
\ge\frac{e^{-\sigma y}W^{p+\alpha}}{q\mathcal J}
\ge\frac{e^{-\alpha\sigma y}W^{p+\alpha}}{q\mathcal J}
=\frac{U^{p+\alpha}}{q\mathcal I}.
\]
By the definition of $V_k$ and $\partial_s\mathcal I=U^p$,
\[
kV_k^q
=
\frac{1}{(kM^q)^{-1}+q\mathcal I}
\le
\frac{1}{q\mathcal I},
\qquad
\partial_sV_k=-kU^pV_k^{q+1}.
\]
Therefore,
\[
\begin{cases}
\partial_sU
\le
\Delta U-kU^{p+\alpha}V_k^q,\\
\partial_sV_k
=
-kU^pV_k^{q+1}.
\end{cases}
\]
Thus $(U,V_k)$ satisfies the differential inequalities
for the lower pair in Theorem~\ref{prop:coupled-comparison}.

Fix $k>0$ and set
\[
\underline u(x,t)=U(x,t-t_0),
\qquad
\underline v(x,t)=V_k(x,t-t_0).
\]
By the definitions of $\underline u$ and $\underline v$ and
Theorem~\ref{thm:general-wellposedness}, the four functions
$\underline u$, $\underline v$, $u_k$, and $v_k$ are nonnegative and
bounded, and have the regularity required in
Theorem~\ref{prop:coupled-comparison}.
On $\Gamma_\delta\times(t_0,T)$, we have
$\underline u=B\le u_k$.
On $\partial\Omega\times(t_0,T)$, $y=\delta$ and
$\partial_\nu y=1$, so
\[
\partial_\nu\underline u
=e^{-\sigma\delta}
\bigl(\partial_yW(\delta,t-t_0)-\sigma W(\delta,t-t_0)\bigr)
\le0=\partial_\nu u_k.
\]
Since $u_k\ge0$, $v_k\le M$, and $0\le\mathcal I(x,s)\le B^ps$,
the definitions and dominated convergence give
\[
\|(\underline u(t)-u_k(t))_+\|_{L^2(\mathcal C_\delta)}
+
\|(v_k(t)-\underline v(t))_+\|_{L^2(\mathcal C_\delta)}
\longrightarrow0
\qquad(t\downarrow t_0).
\]

The proof of Theorem~\ref{prop:coupled-comparison} applies with these
two boundary orderings.
Indeed, let $P=(\underline u-u_k)_+$.
Integration by parts gives the boundary term
\[
\int_{\partial\mathcal C_\delta}
P\,\partial_\nu(\underline u-u_k)\,dS.
\]
On $\Gamma_\delta$, we have $\underline u\le u_k$, and hence $P=0$.
On $\partial\Omega$,
\[
\partial_\nu(\underline u-u_k)\le0,
\]
so that
\[
P\,\partial_\nu(\underline u-u_k)\le0.
\]
Thus the boundary contribution is nonpositive.
The same energy estimate and Gronwall's inequality yield
$\underline u\le u_k$ almost everywhere.
By continuity,
\[
u_k(x,t)\ge U(x,t-t_0)
\quad\text{on }\overline{\mathcal C_\delta}\times(t_0,T].
\]

For $0\le y\le\delta=L/2$ and $t\in[\tau,T]$,
\[
U(x,t-t_0)
\ge
\frac B4
\exp\left(
-\sigma\delta
-\frac{c_w\delta}{\sqrt{\tau/2}}
-\frac{a_w\delta^2}{\tau/2}
\right)
=:\rho_{\tau,T}>0.
\]
The constant $\rho_{\tau,T}$ is independent of $k$ and satisfies
$\rho_{\tau,T}\le B_0$.
Combining the estimates on $\overline{\mathcal C_\delta}$ and
$K_\delta$ proves the first inequality in
\eqref{eq:neumann-global-bounds}.

Applying the first inequality in \eqref{eq:neumann-global-bounds}
with $\tau/2$ in place of $\tau$, choose $\rho>0$,
independent of $k$, such that
\[
u_k\ge\rho
\quad\text{on }\overline\Omega\times[\tau/2,T]
\]
for every $k>0$.
For $(x,t)\in\overline\Omega\times[\tau,T]$,
\[
A_k(x,t)
\ge\int_{\tau/2}^{\tau}u_k(x,s)^p\,ds
\ge\frac{\tau\rho^p}{2}.
\]
Hence \eqref{eq:exact-memory-representation}, extended to the boundary
by continuity, gives
\[
0\le R_k(x,t)
\le\frac{1}{qA_k(x,t)}
\le\frac{2}{q\tau\rho^p}
\quad\text{on }\overline\Omega\times[\tau,T].
\]
The lemma follows.
\end{proof}

We now identify the limit as $k\to\infty$ and prove
Theorem~\ref{thm:memory-limit}.

\begin{proof}[Proof of Theorem~\ref{thm:memory-limit}]
The set $D$ is open because $v_0$ is continuous.
By Lemma~\ref{lem:reaction-rate-monotonicity}, the family
$\{u_k\}_{k>0}$ is nonincreasing with respect to $k$.
Set $M_u=\|u_0\|_{L^\infty(\Omega)}$.
By this monotonicity and \eqref{eq:general-basic-bounds}, the pointwise limit
\[
u(x,t):=\lim_{k\to\infty}u_k(x,t)
\]
exists on $\Omega\times(0,\infty)$ and satisfies
$0\le u\le M_u$.
Moreover, \eqref{eq:main-positive-lower-bound} gives
\[
u(x,t)>0
\qquad
\text{for every }(x,t)\in\Omega\times(0,\infty).
\]

Fix a compact set $K\Subset\Omega$ and $0<\tau<T<\infty$.
Choose bounded open sets $G_0$ and $G_1$ with smooth boundaries such that
\[
K\Subset G_0\Subset G_1\Subset\Omega.
\]
By \eqref{eq:positive-time-reaction-bound} and
\eqref{eq:general-basic-bounds}, the functions $R_k$ and
$R_ku_k^{p+\alpha}$ are uniformly bounded on
$G_1\times(\tau/4,3T)$.
Since \eqref{eq:transformed-critical-system} gives
\[
\partial_tu_k-\Delta u_k=-R_ku_k^{p+\alpha},
\]
the interior parabolic $L^s$-estimates yield
\begin{equation}\label{eq:uniform-interior-parabolic-bound}
\|u_k\|_{W^{2,1}_s(G_0\times(\tau/2,2T))}
\le C_s
\qquad(1<s<\infty),
\end{equation}
where $C_s$ is independent of $k$.
Taking $s>N+2$, the parabolic Sobolev embedding and the
Arzel\`a--Ascoli theorem yield relative compactness in
$C(K\times[\tau,T])$.
Since the pointwise limit is $u$, every subsequential limit equals $u$.
Hence $u_k\to u$ locally uniformly in $\Omega\times(0,\infty)$, with
$u\in C(\Omega\times(0,\infty))$.

Let $A$ be defined by \eqref{eq:A-limit-definition}.
For $(x,t)\in K\times[\tau,T]$ and $0<\varepsilon<\tau$,
\[
\begin{aligned}
|A_k(x,t)-A(x,t)|
&\le
\int_0^\varepsilon
|u_k(x,s)^p-u(x,s)^p|\,ds\\
&\quad+
\int_\varepsilon^t
|u_k(x,s)^p-u(x,s)^p|\,ds\\
&\le
M_u^p\varepsilon
+
T\sup_{(x,s)\in K\times[\varepsilon,T]}
|u_k(x,s)^p-u(x,s)^p|.
\end{aligned}
\]
Hence
\[
\begin{aligned}
\sup_{(x,t)\in K\times[\tau,T]}
|A_k(x,t)-A(x,t)|
&\le
M_u^p\varepsilon\\
&\quad+
T\sup_{(x,s)\in K\times[\varepsilon,T]}
|u_k(x,s)^p-u(x,s)^p|.
\end{aligned}
\]
Since $u_k\to u$ uniformly on $K\times[\varepsilon,T]$,
the second term tends to zero as $k\to\infty$.
Letting $\varepsilon\downarrow0$, we obtain
\begin{equation}\label{eq:uniform-history-convergence}
A_k\longrightarrow A
\quad\text{uniformly on }K\times[\tau,T].
\end{equation}
The positivity and continuity of $u$ imply
\[
A(x,t)>0
\qquad
\text{for every }(x,t)\in\Omega\times(0,\infty).
\]

If $v_0(x)=0$, then
\eqref{eq:exact-memory-representation} gives $R_k(x,t)=0$.
If $v_0(x)>0$, \eqref{eq:exact-memory-representation} and
\eqref{eq:uniform-history-convergence} give
\[
R_k(x,t)
=
\frac{1}{k^{-1}v_0(x)^{-q}+qA_k(x,t)}
\longrightarrow
\frac{1}{qA(x,t)}.
\]
Therefore,
\[
R_k(x,t)\longrightarrow
R(x,t):=\frac{\ind{D}(x)}{qA(x,t)}
\quad
\text{for every }(x,t)\in\Omega\times(0,\infty).
\]
By \eqref{eq:positive-time-reaction-bound} and dominated
convergence, \eqref{eq:Rk-limit-main} holds for every
$1\le s<\infty$.
Since $k^{1/q}v_k=R_k^{1/q}$, the same argument proves
\eqref{eq:scaled-v-limit-main}.
The local uniform convergence of $u_k$ and
\eqref{eq:Rk-limit-main} also yield
\begin{equation}\label{eq:reaction-term-convergence}
R_ku_k^{p+\alpha}
\longrightarrow
\ind{D}\frac{u^{p+\alpha}}{qA}
\quad\text{in }L^1_{\mathrm{loc}}(\Omega\times(0,\infty)).
\end{equation}

For $\phi\in C_c^\infty(\Omega\times(0,\infty))$,
\eqref{eq:transformed-critical-system} yields
\[
\int_{\Omega\times(0,\infty)}
\left(
-u_k\partial_t\phi-u_k\Delta\phi
+R_ku_k^{p+\alpha}\phi
\right)\,dx\,dt=0.
\]
Passing to the limit by local uniform convergence and
\eqref{eq:reaction-term-convergence} proves
\eqref{eq:limit-equation}.

By \eqref{eq:uniform-interior-parabolic-bound} and weak compactness,
we obtain
\[
u\in W^{2,1}_{s,\mathrm{loc}}(\Omega\times(0,\infty))
\qquad\text{for every }1<s<\infty,
\]
and hence also for $s=1$.

It remains to prove convergence up to the boundary and identify
the boundary condition of $u$.
Fix $0<\tau<T<\infty$.
In the Dirichlet case, Lemma~\ref{lem:reaction-rate-monotonicity}
and \eqref{eq:general-basic-bounds} give
\[
0\le u\le u_k\le u_1
\quad\text{in }\Omega\times(0,\infty),\qquad k\ge1.
\]
Since $u_1$ is continuous up to the boundary and satisfies
\eqref{eq:dirichlet-bc}, $u$ extends continuously by zero to
$\partial\Omega\times(0,\infty)$, proving
\eqref{eq:dirichlet-limit-boundary}.
For every $\varepsilon>0$, choose $\delta>0$ such that
\[
u_1(x,t)<\varepsilon
\quad\text{if }\operatorname{dist}(x,\partial\Omega)<\delta,
\quad t\in[\tau,T].
\]
The inequality $0\le u_k-u\le u_1$ controls the difference in this
boundary neighborhood, while local uniform convergence controls it
on the compact set
\[
K_\delta:=\{x\in\overline\Omega:
\operatorname{dist}(x,\partial\Omega)\ge\delta\}.
\]
Thus $u_k\to u$ uniformly on $\overline\Omega\times[\tau,T]$.

In the Neumann case, Lemma~\ref{lem:neumann-boundary-estimates} and
\eqref{eq:general-basic-bounds} imply that $R_ku_k^{p+\alpha}$ is
uniformly bounded on $\Omega\times(\tau/4,2T)$.
Choose $\zeta\in C^\infty([\tau/4,2T])$ such that
\[
0\le\zeta\le1,
\qquad
\zeta=0 \quad\text{on }[\tau/4,3\tau/8],
\qquad
\zeta=1 \quad\text{on }[\tau/2,2T].
\]
Then $w_k=\zeta u_k$ satisfies
\[
\begin{cases}
\partial_tw_k-\Delta w_k
=-\zeta R_ku_k^{p+\alpha}+\zeta'u_k
&\text{in }\Omega\times(\tau/4,2T),\\
\partial_\nu w_k=0
&\text{on }\partial\Omega\times(\tau/4,2T),\\
w_k(\cdot,\tau/4)=0&\text{in }\Omega.
\end{cases}
\]
The global parabolic $L^s$-estimate for the homogeneous Neumann
problem yields
\begin{equation}\label{eq:neumann-global-parabolic-bound}
\|u_k\|_{W^{2,1}_s(\Omega\times(\tau/2,2T))}
\le C_{s,\tau,T}\qquad(1<s<\infty),
\end{equation}
where $C_{s,\tau,T}$ is independent of $k$.
By \eqref{eq:neumann-global-parabolic-bound}, local uniform
convergence, and weak compactness, \eqref{eq:neumann-global-regularity}
holds for $1<s<\infty$, and hence also for $s=1$.

Taking $s>N+2$, the parabolic Sobolev embedding and the
Arzel\`a--Ascoli theorem yield relative compactness in
$C([\tau,T];C^1(\overline\Omega))$.
Together with the local uniform convergence of $u_k$ to $u$, this proves
\eqref{eq:neumann-c1-convergence}, and in particular uniform convergence
on $\overline\Omega\times[\tau,T]$.
Since $\partial_\nu u_k=0$, \eqref{eq:neumann-c1-convergence} yields
\eqref{eq:neumann-limit-boundary}.

It remains to identify the initial trace.
Let $\widehat u$ be the solution of the heat equation with initial datum
$u_0$ and the same homogeneous boundary condition as $u_k$.
Since
\[
\partial_t u_k-\Delta u_k\le0,
\]
the comparison principle gives
\[
0\le u_k(x,t)\le\widehat u(x,t),
\qquad
0\le u(x,t)\le\widehat u(x,t).
\]
If $u_0(x)=0$, then $0\le u(x,t)\le\widehat u(x,t)\to0$ as
$t\downarrow0$.
If $u_0(x)>0$, then, by the continuity of $u_0$, we can choose a ball
$B\Subset\{u_0>0\}$ containing $x$ and
$\xi\in C_c^\infty(B)$ with $0\le\xi\le1$ and $\xi=1$ near $x$.
Since $v_0=0$ on $B$, we have $v_k=0$ on $B$ for all $t\ge0$.
Let $u_B$ be the solution of the Dirichlet heat equation on $B$ with
initial datum $\xi u_0$. By the comparison principle,
\[
u_B\le u_k\le\widehat u
\]
on $B$. Letting first $k\to\infty$ and then $t\downarrow0$, we obtain
\[
u(x,t)\to u_0(x).
\]
Hence, by the dominated convergence theorem,
\[
u(\cdot,t)\to u_0
\qquad\text{in }L^2(\Omega)
\quad\text{as }t\downarrow0,
\]
which proves \eqref{eq:limit-initial-trace}.

We also record the global integrability of the limiting reaction term,
which will be used in Section~\ref{sec:memory-uniqueness}. Integrating
the first equation of \eqref{eq:Pk-intro} over $\Omega\times(0,T)$ and
using the boundary condition, we obtain
\[
\int_0^T\!\int_\Omega
k u_k^{p+\alpha}v_k^q\,dx\,dt
\le \int_\Omega u_0\,dx.
\]
Since
\[
k u_k^{p+\alpha}v_k^q
\longrightarrow
\ind{D}\frac{u^{p+\alpha}}{qA}
\]
pointwise for $t>0$, Fatou's lemma yields
\begin{equation}\label{eq:limit-reaction-global-L1}
\int_0^T\!\int_\Omega
\ind{D}\frac{u^{p+\alpha}}{qA}\,dx\,dt
\le \int_\Omega u_0\,dx
\qquad (T>0).
\end{equation}
The theorem follows.
\end{proof}

\section{Uniqueness for separated initial supports}
\label{sec:memory-uniqueness}

In this section, we prove Theorem~\ref{thm:memory-uniqueness}.
We first establish a comparison principle for \eqref{eq:limit-equation},
which will be used in the uniqueness argument.

\begin{lemma}\label{lem:memory-comparison}
Let $G\subseteq\Omega$ be an open set with $C^2$ boundary, and let
$t_0<t_1$. Let $w_1,w_2$ be nonnegative bounded functions such that
\[
w_1,w_2\in W^{2,1}_2(G\times(s,t_1))
\qquad(t_0<s<t_1).
\]
For $i=1,2$, let
\[
A_i(x,t)=A_i(x,t_0)+\int_{t_0}^t w_i(x,s)^p\,ds,
\]
where $A_i(\cdot,t_0)\in L^\infty(G)$ and
\[
0\le A_1(\cdot,t_0)\le A_2(\cdot,t_0).
\]
Assume that $A_1,A_2>0$ on $(D\cap G)\times(t_0,t_1)$ and
\[
\partial_t w_1-\Delta w_1
\le-\ind{D}\frac{w_1^{p+\alpha}}{qA_1},
\qquad
\partial_t w_2-\Delta w_2
\ge-\ind{D}\frac{w_2^{p+\alpha}}{qA_2}
\]
almost everywhere in $G\times(t_0,t_1)$, where the right-hand sides are
understood to be zero outside $D$. Suppose also that, for some $C>0$,
\begin{equation}\label{eq:memory-comparison-ratio}
w_2^p\le CA_2
\quad\text{on }(D\cap G)\times(t_0,t_1).
\end{equation}
On each connected component of $\partial G$, assume either
\[
w_1\le w_2
\qquad\text{or}\qquad
\partial_\nu w_1\le\partial_\nu w_2
\]
for $t_0<t<t_1$. If
\[
\|(w_1(t)-w_2(t))_+\|_{L^2(G)}
\longrightarrow0
\qquad(t\downarrow t_0),
\]
then
\[
w_1\le w_2,
\qquad
A_1\le A_2
\]
almost everywhere in $G\times(t_0,t_1)$.
\end{lemma}

\begin{proof}
Set $P=(w_1-w_2)_+$.
Since $w_1\le w_2+P$ and $A_1(\cdot,t_0)\le A_2(\cdot,t_0)$,
Minkowski's inequality gives
\begin{equation}\label{eq:root-history-difference}
(A_1(x,t)^{1/p}-A_2(x,t)^{1/p})_+
\le
\left(\int_{t_0}^t P(x,s)^p\,ds\right)^{1/p}.
\end{equation}

On $D\cap\{P>0\}$, if $A_1\le A_2$, then
\[
\frac{w_2^{p+\alpha}}{qA_2}
-\frac{w_1^{p+\alpha}}{qA_1}\le0.
\]
If $A_1>A_2$, then
\[
\begin{aligned}
\frac{w_2^{p+\alpha}}{qA_2}
-\frac{w_1^{p+\alpha}}{qA_1}
&\le
\frac{w_2^{p+\alpha}}{q}
\left(\frac1{A_2}-\frac1{A_1}\right)\\
&\le
\frac{p\,w_2^{p+\alpha}}{qA_2^{1+1/p}}
\left(A_1^{1/p}-A_2^{1/p}\right)\\
&\le
C\left(\int_{t_0}^t P(x,s)^p\,ds\right)^{1/p}.
\end{aligned}
\]
Here the second inequality follows from the mean value theorem applied to
$r\mapsto r^{-p}$, and
\[
\frac{w_2^{p+\alpha}}{A_2^{1+1/p}}
=
\left(\frac{w_2^p}{A_2}\right)^{1+1/p}
w_2^{\alpha-1}
\le C
\]
by \eqref{eq:memory-comparison-ratio}, $\alpha\ge1$, and the boundedness
of $w_2$. The last inequality follows from
\eqref{eq:root-history-difference}.

Set
\[
m(t)=\mathop{\mathrm{ess\,sup}}_{G\times(t_0,t)}P,
\qquad
\eta(t)=C\int_{t_0}^t(s-t_0)^{1/p}m(s)\,ds.
\]
Then
\[
\partial_t(w_1-w_2)-\Delta(w_1-w_2)
\le\eta'(t)
\quad\text{on }\{P>0\}.
\]
Set $z=w_1-w_2-\eta$.
Since $\eta\ge0$, we have $\{z>0\}\subset\{P>0\}$, and hence
\[
\partial_tz-\Delta z\le0
\quad\text{on }\{z>0\}.
\]
Using the boundary assumptions, multiplying by $z_+$ and integrating
over $G\times(s,t)$, we obtain
\[
\frac12\|z_+(t)\|_{L^2(G)}^2
+\int_s^t\!\int_G|\nabla z_+|^2\,dx\,d\tau
\le
\frac12\|z_+(s)\|_{L^2(G)}^2.
\]
Since $\eta(s)\ge0$,
\[
z_+(s)
\le (w_1(s)-w_2(s))_+.
\]
Thus
\[
\|z_+(t)\|_{L^2(G)}
\le
\|(w_1(s)-w_2(s))_+\|_{L^2(G)}.
\]
Letting $s\downarrow t_0$ and using the initial ordering, we obtain
$z_+(t)=0$. Therefore
\[
P(\cdot,t)\le\eta(t)
\quad\text{a.e. in }G.
\]
Since $\eta$ is nondecreasing, we have $m(t)\le\eta(t)$, and therefore
\[
\eta'(t)\le C(t-t_0)^{1/p}\eta(t),
\qquad
\eta(t_0)=0.
\]
Gronwall's inequality gives $\eta=0$, and hence $w_1\le w_2$ and $A_1\le A_2$.
The lemma follows.
\end{proof}

To apply Lemma~\ref{lem:memory-comparison}, we need to verify
\eqref{eq:memory-comparison-ratio}. The next lemma shows that the ratio
$w^p/A_w$ is bounded away from the initial time.

\begin{lemma}\label{lem:memory-history-ratio}
Let $w$ be either the limit in Theorem~\ref{thm:memory-limit} or a
function satisfying the conditions in
Theorem~\ref{thm:memory-uniqueness}.
For every $0<\tau<T<\infty$, there is $C_{\tau,T}>0$ such that
\begin{equation}\label{eq:memory-positive-time-ratio}
w(x,t)^p\le C_{\tau,T}A_w(x,t)
\qquad(x\in\Omega,\ \tau\le t\le T).
\end{equation}
\end{lemma}

\begin{proof}
Let $\widehat u$ be the solution of the heat equation with initial datum
$u_0$ and the same homogeneous boundary condition as $w$.
Since $\partial_t w-\Delta w\le0$, the comparison principle, together with
\[
w(\cdot,t)\to u_0
\quad\text{in }L^2(\Omega)
\quad\text{as }t\downarrow0,
\]
gives
\begin{equation}\label{eq:memory-basic-bound}
0\le w\le\widehat u\le\|u_0\|_{L^\infty(\Omega)}.
\end{equation}

Write
\[
d(x)=\operatorname{dist}(x,\partial\Omega),
\qquad
t_0=\frac{\tau}{2}.
\]
Choose $\delta>0$ so that $d$ is $C^2$ on $\{0<d<2\delta\}$, and set
$\sigma=\|\Delta d\|_{L^\infty(\{0<d<\delta\})}$.
For either choice of $w$, we have
\[
w>0
\quad\text{in }\Omega\times(0,\infty),
\]
and $w$ is continuous for positive times. Hence, choose $0<B\le1$ such that
\[
w\ge B
\quad\text{on }\{d=\delta\}\times[t_0,T].
\]
Let $W$ be the barrier in
Lemma~\ref{lem:one-dimensional-memory-barrier}, with this value of $B$
and $L=2\delta$. For $0<r<\delta/2$ and $r<d(x)<\delta$, define
\[
\begin{aligned}
U_r(x,t)
&=\frac{d(x)-r}{\delta-r}
e^{-\sigma(\delta-d(x))}
W(\delta-d(x),t-t_0),\\
I_r(x,t)
&=\int_{t_0}^t U_r(x,s)^p\,ds.
\end{aligned}
\]
Using $|\nabla d|=1$, $|\Delta d|\le\sigma$, $\partial_y W\le0$, and
$\alpha\ge1$, direct differentiation as in the proof of
Lemma~\ref{lem:neumann-boundary-estimates} gives
\[
\partial_tU_r-\Delta U_r
\le-\frac{U_r^{p+\alpha}}{qI_r}
\le-\ind{D}\frac{U_r^{p+\alpha}}{qI_r}.
\]

For fixed $r>0$, $A_w$ has a positive lower bound on
$\{r\le d\le\delta\}\times[t_0,T]$. Together with
\eqref{eq:memory-basic-bound}, this gives
\eqref{eq:memory-comparison-ratio} for $(w,A_w)$ on this set.
Moreover,
\[
I_r(\cdot,t_0)=0\le A_w(\cdot,t_0),
\]
$U_r\le w$ on $d=r,\delta$, and
\[
U_r(\cdot,t)\longrightarrow0
\quad\text{in }L^2(\{r<d<\delta\})
\quad\text{as }t\downarrow t_0.
\]
Lemma~\ref{lem:memory-comparison} therefore gives $U_r\le w$.
Letting $r\downarrow0$, we obtain
$w(x,t)\ge c_{\tau,T}d(x)$ near $\partial\Omega$.
Since $w>0$ and is continuous on
$\{d\ge\delta\}\times[\tau,T]$, after decreasing $c_{\tau,T}$ if
necessary,
\begin{equation}\label{eq:memory-distance-lower}
w(x,t)\ge c_{\tau,T}d(x)
\qquad
(x\in\Omega,\ \tau\le t\le T).
\end{equation}

In the Dirichlet case, positive-time boundary estimates for
$\widehat u$ and \eqref{eq:memory-basic-bound} give
\begin{equation}\label{eq:memory-dirichlet-upper}
w(x,t)\le C_{\tau,T}d(x)
\qquad
(x\in\Omega,\ \tau\le t\le T).
\end{equation}
Applying \eqref{eq:memory-distance-lower} with $\tau/2$ in place of
$\tau$, for $t\ge\tau$ we obtain
\begin{equation}\label{eq:memory-dirichlet-history-lower}
A_w(x,t)
\ge
\int_{\tau/2}^{\tau}w(x,s)^p\,ds
\ge
c_{\tau,T}d(x)^p.
\end{equation}
Hence \eqref{eq:memory-positive-time-ratio} follows from
\eqref{eq:memory-dirichlet-upper} and
\eqref{eq:memory-dirichlet-history-lower}.

In the Neumann case, if $w$ vanished at a boundary point at positive
time, then $C^1$ regularity and $\partial_\nu w=0$ would give
\[
w(x-\rho\nu(x),t)=o(\rho)
\qquad(\rho\downarrow0),
\]
contradicting \eqref{eq:memory-distance-lower}. Hence, applying
\eqref{eq:memory-distance-lower} with $\tau/2$ in place of $\tau$ and
using compactness, we obtain
\begin{equation}\label{eq:memory-neumann-positive-lower}
w(x,t)\ge c_{\tau,T}>0
\qquad
(x\in\overline\Omega,\ \tau/2\le t\le T).
\end{equation}
Therefore, for $t\ge\tau$,
\begin{equation}\label{eq:memory-neumann-history-lower}
A_w(x,t)
\ge
\int_{\tau/2}^{\tau}w(x,s)^p\,ds
\ge
c_{\tau,T}>0.
\end{equation}
Together with \eqref{eq:memory-basic-bound}, this gives
\eqref{eq:memory-positive-time-ratio}. The lemma follows.
\end{proof}

By \eqref{eq:memory-positive-time-ratio} and
\eqref{eq:memory-basic-bound},
\[
\ind{D}\frac{u^{p+\alpha}}{qA}
=
\frac{\ind{D}}{q}u^\alpha\frac{u^p}{A}
\le C_{\tau,T}
\qquad
\text{on }\Omega\times[\tau,T].
\]
In the Dirichlet case, applying this estimate with $\tau/2$ in place of
$\tau$, we see that $u$ solves a linear parabolic equation with bounded
right-hand side on $\Omega\times(\tau/2,T)$ and zero boundary data.
Standard parabolic estimates therefore give
$u\in W^{2,1}_2(\Omega\times(\tau,T))$.
In the Neumann case, this regularity is already contained in
Theorem~\ref{thm:memory-limit}. Moreover,
\eqref{eq:limit-reaction-global-L1} gives
\[
\ind{D}\frac{u^{p+\alpha}}{qA}
\in L^1(\Omega\times(0,T)).
\]
Hence $u$ satisfies all the conditions in
Theorem~\ref{thm:memory-uniqueness}.

We next compare two solutions for small times in a neighborhood of the
set where $u_0$ is nonzero.

\begin{lemma}\label{lem:short-time-boundary-order}
Assume \eqref{eq:separated-initial-supports}.
There is a bounded open set $\Omega_0$ with smooth boundary such that
\[
\mathcal K\Subset\Omega_0\Subset\Omega\setminus\overline D
\]
and the following holds.
For any two functions $w_1,w_2$ satisfying the conditions in
Theorem~\ref{thm:memory-uniqueness} and any $c>1$, we have
\begin{equation}\label{eq:uniqueness-lateral-order}
w_1(x,t)\le c w_2(x,t+h)
\qquad
(x\in\partial\Omega_0,\ 0<t\le s,\ 0<h<s)
\end{equation}
for all sufficiently small $s>0$.
\end{lemma}

\begin{proof}
Let $w_i$, $i=1,2$, be as in the statement.
Choose $r_0>0$ so small that
\[
\{x:\operatorname{dist}(x,\mathcal K)\le3r_0\}
\subset\Omega\setminus\overline D,
\]
and let
\[
\widetilde u(x,t)=(4\pi t)^{-N/2}
\int_\Omega e^{-|x-y|^2/(4t)}u_0(y)\,dy.
\]
For $\operatorname{dist}(x,\mathcal K)<r_0$, set
\[
a=\frac{r_0}{\sqrt N},
\qquad
Q_x=\{y\in\mathbb R^N:|y_j-x_j|<a,\ j=1,\ldots,N\}.
\]
Then $Q_x\Subset\Omega\setminus\overline D$, so $w_i$ and
$\widetilde u$ solve the heat equation in $Q_x$, have the same
$L^2$ initial trace, and satisfy
\[
0\le w_i,\widetilde u\le\|u_0\|_{L^\infty(\Omega)}.
\]

For $\lambda>0$, set
\[
\Phi_\lambda(y,t)
=
\|u_0\|_{L^\infty(\Omega)}e^{\lambda^2t}
\sum_{j=1}^N
\left(
e^{-\lambda(a-y_j+x_j)}
+
e^{-\lambda(a+y_j-x_j)}
\right).
\]
Since $\Phi_\lambda$ solves the heat equation and
\[
\Phi_\lambda\ge
\|u_0\|_{L^\infty(\Omega)}
\ge |w_i-\widetilde u|
\quad\text{on }\partial Q_x,
\]
the comparison principle gives
\[
|w_i(y,t)-\widetilde u(y,t)|
\le\Phi_\lambda(y,t).
\]
Evaluating at $y=x$ and choosing $\lambda=a/(2t)$, we obtain
\begin{equation}\label{eq:empty-region-heat-error}
|w_i(x,t)-\widetilde u(x,t)|
\le
2N\|u_0\|_{L^\infty(\Omega)}
e^{-r_0^2/(4Nt)}.
\end{equation}

Choose $0<r_1<r_0/(3\sqrt N)$ and a smooth open set $\Omega_0$ with
\[
\{x:\operatorname{dist}(x,\mathcal K)\le r_1\}
\Subset\Omega_0
\Subset\{x:\operatorname{dist}(x,\mathcal K)<2r_1\}.
\]
For $x\in\partial\Omega_0$, the continuous function
$\int_{B_{3r_1}(x)\cap\Omega}u_0(y)\,dy$ has a positive minimum.
Consequently, for some $c_0>0$,
\[
\begin{aligned}
\widetilde u(x,t)
&\ge
(4\pi t)^{-N/2}
\int_{B_{3r_1}(x)\cap\Omega}
e^{-|x-y|^2/(4t)}u_0(y)\,dy\\
&\ge
c_0(4\pi t)^{-N/2}e^{-9r_1^2/(4t)}
\qquad(x\in\partial\Omega_0).
\end{aligned}
\]
Together with \eqref{eq:empty-region-heat-error}, this gives
\[
\frac{|w_i(x,t)-\widetilde u(x,t)|}{\widetilde u(x,t)}
\longrightarrow0
\]
uniformly for $x\in\partial\Omega_0$ as $t\downarrow0$. Hence
$w_i/\widetilde u\to1$
uniformly on $\partial\Omega_0$.

Since $|x-y|\ge r_1$ for $x\in\partial\Omega_0$ and
$y\in\mathcal K$, differentiation of the Gaussian shows that
$t\mapsto\widetilde u(x,t)$ is nondecreasing for
$0<t\le r_1^2/(2N)$.
Thus, for sufficiently small $s>0$,
\[
w_1(x,t)\le\sqrt c\,\widetilde u(x,t)
\le\sqrt c\,\widetilde u(x,t+h)\le c w_2(x,t+h)
\]
when $x\in\partial\Omega_0$, $0<t\le s$, and $0<h<s$. The lemma follows.
\end{proof}

\begin{proof}[Proof of Theorem~\ref{thm:memory-uniqueness}]
Let $w_1,w_2$ satisfy the conditions in the theorem, and write
$A_i=A_{w_i}$ for $i=1,2$.
Fix $T>0$ and $c>1$.
Take $\Omega_0$ from Lemma~\ref{lem:short-time-boundary-order}, and
set $\Omega_1=\Omega\setminus\overline{\Omega_0}$.
Choose $0<s<T/2$ so that \eqref{eq:uniqueness-lateral-order} holds.
For $0<h<s$, set
\[
\overline w(x,t)=c w_2(x,t+h),
\qquad
\overline A(x,t)=c^p A_2(x,t+h).
\]
Then $\partial_t\overline A=\overline w^p$ and
\[
\partial_t\overline w-\Delta\overline w
+\ind{D}\frac{\overline w^{p+\alpha}}{q\overline A}
=(c^\alpha-c)\ind{D}
\frac{w_2(x,t+h)^{p+\alpha}}{qA_2(x,t+h)}\ge0.
\]
On $\Omega_1$, $w_1$ has zero initial trace, and
$A_1(\cdot,0)=0\le\overline A(\cdot,0)$.
The boundary ordering on $\partial\Omega_0$ follows from
\eqref{eq:uniqueness-lateral-order}. On $\partial\Omega$, both functions
satisfy the same homogeneous boundary condition.
For fixed $h>0$, Lemma~\ref{lem:memory-history-ratio} gives
$\overline w^p\le C_h\overline A$ on $\Omega_1\times(0,s)$.
Lemma~\ref{lem:memory-comparison} therefore yields
$w_1(x,t)\le c w_2(x,t+h)$ there.
Letting $h\downarrow0$ gives $w_1\le c w_2$ on
$\Omega_1\times(0,s]$.

On $\Omega_0$, $w_1$ and $c w_2$ solve the heat equation.
Their initial traces are $u_0$ and $c u_0$, respectively, and their
boundary values are ordered by the preceding inequality and continuity.
The comparison principle gives $w_1\le c w_2$ also on
$\Omega_0\times(0,s]$.
Hence
\[
w_1(\cdot,s)\le c w_2(\cdot,s),
\qquad A_1(\cdot,s)\le c^p A_2(\cdot,s).
\]
Apply Lemma~\ref{lem:memory-comparison} once more on
$\Omega\times(s,T)$ to $(w_1,A_1)$ and $(c w_2,c^p A_2)$,
using Lemma~\ref{lem:memory-history-ratio}.
We obtain $w_1\le c w_2$ on $\Omega\times(0,T]$.
Letting $c\downarrow1$ and interchanging $w_1,w_2$ proves uniqueness.
Since $T>0$ was arbitrary, the theorem follows.
\end{proof}

\section*{Acknowledgments}
This work was supported in part by JSPS KAKENHI Grant Numbers JP23H00085 and JP26K17020.


\end{document}